\documentclass[a4paper]{article}[12pts]

\usepackage[english]{babel}
\usepackage[utf8x]{inputenc}
\usepackage[T1]{fontenc}

\normalsize

\usepackage[a4paper,top=1in,bottom=1in,left=1in,right=1in,marginparwidth=1in]{geometry}

\usepackage{setspace}
\usepackage{amsmath}
\usepackage{amssymb} 
\usepackage{amsthm}
\usepackage{amsfonts, mathtools}
\usepackage{algorithm,algpseudocode}
\usepackage{bm}
\usepackage{comment}
\usepackage{graphicx}
\usepackage{multirow, booktabs}
\usepackage{caption}
\usepackage{subcaption}
\usepackage[colorinlistoftodos]{todonotes}
\usepackage[colorlinks=true, allcolors=blue]{hyperref}

\newcommand{\E}{\mathbb{E}}

\newcommand{\prob}{\mathbb{P}}

\DeclareMathOperator*{\argmin}{arg\,min}

\newtheorem{proposition}{Proposition}
\newtheorem{lemma}{Lemma}

\newtheorem{assumption}{Assumption}
\newtheorem{theorem}{Theorem}
\newtheorem{corollary}{Corollary}

\usepackage{natbib}

\title{Learning from Stochastically Revealed Preference}
\author{John R. Birge$^\dagger$ \and Xiaocheng Li$^\ddagger$ \and  Chunlin Sun$^\Diamond$}
\date{\small 
$^\dagger$ The University of Chicago Booth School of Business,  John.Birge@chicagobooth.edu\\
$^\ddagger$ Imperial College Business School, Imperial College London, xiaocheng.li@imperial.ac.uk\\
$^\Diamond$ Institute for Computational and Mathematical Engineering, Stanford University, chunlin@stanford.edu
}

\begin{document}
\maketitle

\onehalfspacing

\begin{abstract}
We study the learning problem of revealed preference in a stochastic setting: a learner observes the utility-maximizing actions of a set of agents whose utility follows some unknown distribution, and the learner aims to infer the distribution through the observations of actions. The problem can be viewed as a single-constraint special case of the inverse linear optimization problem. Existing works all assume that all the agents share one common utility which can easily be violated under practical contexts. In this paper, we consider two settings for the underlying utility distribution: a Gaussian setting where the customer utility follows the von Mises-Fisher distribution, and a $\delta$-corruption setting where the customer utility distribution concentrates on one fixed vector with high probability and is arbitrarily corrupted otherwise. We devise Bayesian approaches for parameter estimation and develop theoretical guarantees for the recovery of the true parameter. We illustrate the algorithm performance through numerical experiments.  
\end{abstract}

\section{Introduction}

The problem of learning from revealed preference refers to the learning of a common utility function for a set of agents based on the observations of the utility-maximizing actions from the agents. The revealed preference problem has a long history in economics \citep{samuelson1948consumption, afriat1967construction} (See \citep{varian2006revealed} for a review). A line of works \citep{beigman2006learning,zadimoghaddam2012efficiently,balcan2014learning,amin2015online} formulate the problem as a learning problem with two objectives: (i) rationalizing a
set of observations, i.e., to find a utility function which explains a set of past observations; (ii) predicting the future behavior of a utility-maximization agent. Mathematically, the action of the agents is modeled by an optimization problem that maximizes a linear (or concave) utility function subject to one linear budget constraint. The learner (decision maker) aims to learn the unknown utility function through a set of observations of the constraints and the optimal solutions. The problem can be viewed as a single-constraint special case of the inverse optimization problem \citep{ahuja2001inverse} which covers a wider range of applications: geoscience \citep{burton1992instance}, finance \citep{bertsimas2012inverse},  market analysis \citep{birge2017inverse}, energy \citep{aswani2018inverse}, etc.

In this paper, we study the problem under a stochastic setting where the agents have a linear utility function randomly distributed according to some unknown distribution. Such a stochastic setting is well-motivated by some application context where the agents are customers and the constraint models the prices and the customer's budget. The optimal solution encodes the customer's purchase behavior and the stochastic utility (objective function) captures the heterogeneity of the customer preference for the products. The goal of learning in this stochastic setting thus becomes to learn the utility distribution through observations of the actions. To the best of our knowledge, we provide the first result of learning a stochastic utility for the revealed preference problem and even in the more general literature of the inverse optimization problem. 

\textbf{Related Literature:} The existing approaches to the problem can be roughly divided into two categories.

Query-based: In a query-based model, the
learner aims to learn the utility function by querying an oracle for the agent’s optimal actions, and the goal is to derive the sample complexity guarantee for a sufficiently accurate estimation of the utility function. \cite{beigman2006learning} initiates this line of research and studies a statistical setup where the input data is a set of observations and the learner's performance is evaluated by sample complexity bounds. \cite{zadimoghaddam2012efficiently} studies the case of a linear or linearly separable concave utility function, and \cite{balcan2014learning} generalizes the setting and devises learning algorithms for several classes of utility functions. Other than a statistical setup where the observations are sampled from some distribution, both of these two works study an ``active'' learning setting where the learner has the power to choose the linear constraint (set the prices of the products). Some subsequent works along this line study the associated revenue management problem \citep{amin2015online} and a game-theoretic setting \citep{dong2018strategic} where the agents act strategically to hide the true actions. 

Optimization-based: The optimization-based approach is usually adopted in the literature of inverse optimization, and some algorithms developed therein can be applied to the special case of the revealed preference problem. \citep{zhang1996calculating} and \citep{ahuja2001inverse} study the inverse optimization with one single observation and develop linear programming formulations to solve the problem.
Later, \citep{keshavarz2011imputing} and \citep{aswani2018inverse} study the statistical (or data-driven) setting where the observations are sampled from some distribution. Specifically, \citep{aswani2018inverse} considers a setting where the optimal actions of the agents are contaminated with some independent noises, but all the agents still follow a common utility parameter vector. \citep{mohajerin2018data} studies the distributional robust version of the problem and \citep{besbes2021contextual} considers a contextual formulation. A recent line of works \citep{barmann2018online,dong2018generalized, dong2020expert, chen2020online} cast the inverse optimization problem in an online context and develop (online) gradient-based algorithms. 

As we understand, all the existing algorithms and analyses under this topic rely on the assumption that all the agents share one common utility function (or a common utility parameter vector), and thus can fail in the stochastic setting. In this paper, we formulate the problem in Section \ref{sec_model} and focus on the statistical data input where the budget constraint is sampled from some unknown distribution. We consider two stochastic settings: a Gaussian setting in Section \ref{sec_gaussian} and a $\delta$-corruption setting in Section \ref{sec_delta}. We conclude with numerical experiments and discuss (i) how the results can be generalized to the inverse optimization problem and (ii) the implications for the query-based model where the learner has the power to choose the budget constraints.

\section{Model Setup}

\label{sec_model}

Consider a customer who purchases a bundle of products subject to some budget constraint. The customer's utility-maximizing action can be modeled by the following linear program:
\begin{align*}
   \text{LP}(\bm{u}, \bm{a}, b) \coloneqq \max_{\bm{x}}\  & \sum_{i=1}^n u_ix_i \\
\text{s.t.}\     & \sum_{i=1}^n a_i x_i \le b, \quad 0\le x_i\le 1, i=1,...,n,
\end{align*}
where $\bm{u}=(u_1,...,u_n)\in\mathbb{R}^n$, $\bm{a}=(a_1,...,a_n)\in\mathbb{R}^n_+$, and $b\in\mathbb{R}_+$ are the inputs of the LP. Here the decision variables $\bm{x}$ encode the purchase decisions where a partial purchase is allowed, $u_i$ denotes the customer's utility for the $i$-th product, and $a_i$ denotes the price or cost of purchasing the $i$-th product. The right-hand-side $b$ denotes the budget of the customer. 

Throughout this paper, we make the following assumption. 

\begin{assumption}
We assume:
\begin{itemize}
    \item The utility vector $\bm{u}$ follows some unknown distribution $\mathcal{P}_u$.
    \item The LP's input $(\bm{a}, b)$ follows some unknown distribution  $\mathcal{P}_{\bm{a},b}$ independent of $\mathcal{P}_{u}$. 
    \item There exists $\underline{a}>0$ such that $a_i \in [\underline{a},1]$ almost surely for $i=1,...,n$.
\end{itemize} 
\label{assp_distr}
\end{assumption}

Our goal is to infer the distribution through observations of customers' optimal actions. Mathematically, we aim to estimate the distribution of $\mathcal{P}_u$ through the dataset
$$\mathcal{D}_T = \left\{\left(\bm{x}_{t}^*,\bm{a}_t,b_t\right)\right\}_{t=1}^T.$$
Here the $t$-th sample corresponds to an unobservable $\bm{u}_t$ generated from $\mathcal{P}_u$ and $\bm{x}_t^*$ is one optimal solution of $\text{LP}(\bm{u}_t, \bm{a}_t, b_t).$
Due to the scale invariance of the utility vector, we restrict the distribution $\mathcal{P}_u$ to the unit sphere $\mathcal{S}^{n-1}=\{\bm{u}:\|\bm{u}\|_2=1\}$. In the following two sections, we consider two settings: (i) Gaussian: $\mathcal{P}_u$ follows the von Mises–Fisher distribution, i.e., the restriction of a multivariate Gaussian distribution to the unit sphere; (ii) $\delta$-corruption: $\mathcal{P}_u$ concentrates on one point $\bm{u}^*$ with probability $1-\delta$ and follows an arbitrarily corrupted distribution with probability $\delta$.

\begin{figure}[ht!]
    \centering
    \includegraphics[scale=0.35]{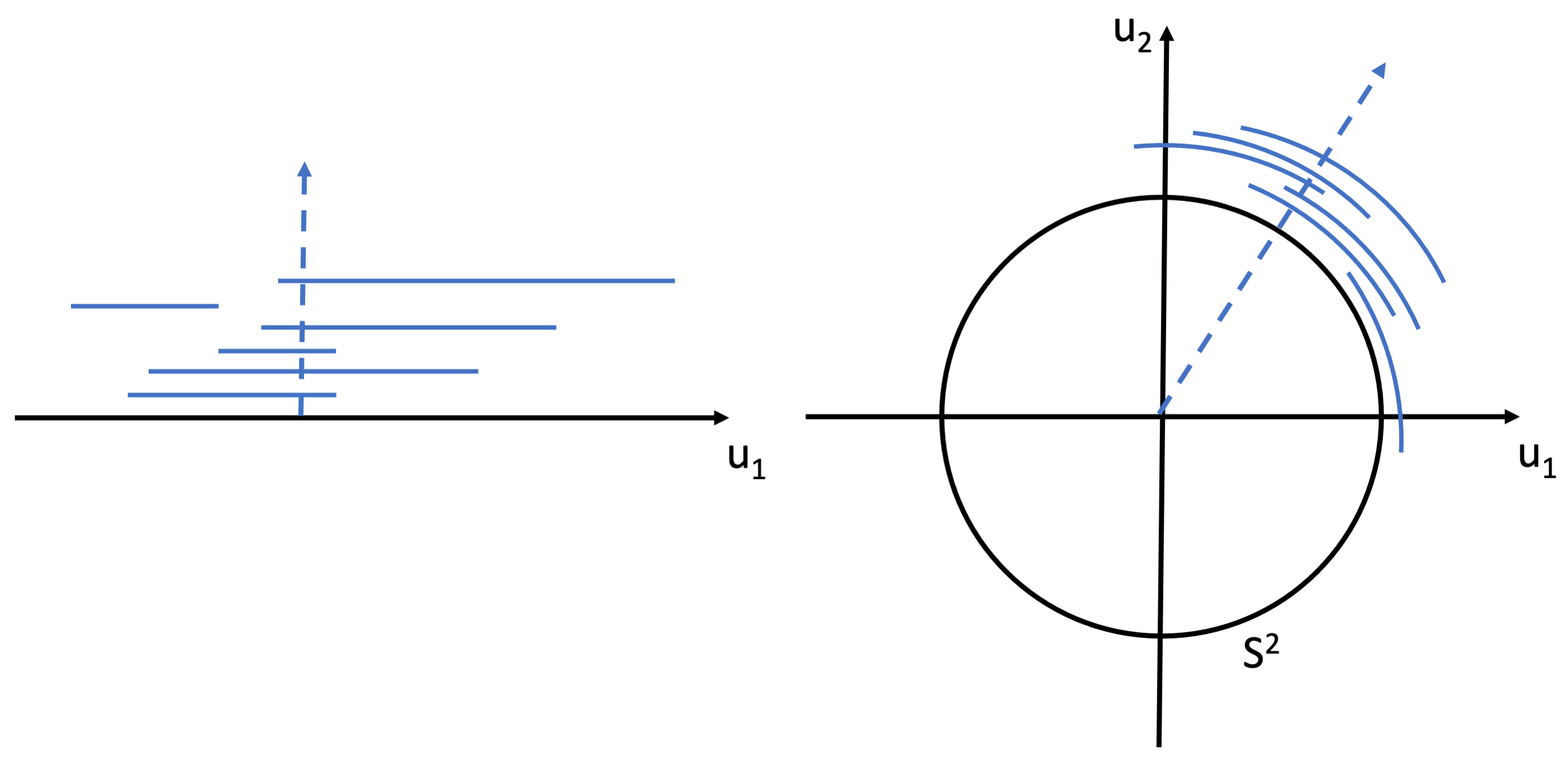}
    \caption{Visualizing the challenge of the problem in 1-D and 2-D.}
    \label{fig:vis}
\end{figure}

\paragraph{The challenge of the problem.} 

Each observation $(\bm{x}_t^*, \bm{a}_t, b_t)$ prescribes a region $\mathcal{U}_t\subset \mathcal{S}^{n-1}$,
$$\mathcal{U}_t\coloneqq\left\{\bm{u}\in \mathcal{S}^{n-1}: \bm{x}_t^* \text{ is an optimal solution of } \text{LP}(\bm{u}, \bm{a}_t, b_t)\right\}.$$
The set $\mathcal{U}_t$ captures all the possible values of $\bm{u}_t$ that is consistent with the $t$-th observation. The following lemma states that the set $\mathcal{U}_t$ can be expressed by a group of linear constraints.

\begin{lemma}
\label{lemma_linear}
For each $\mathcal{U}_t$, there exists a matrix $\bm{V}_t$ and a vector $\bm{w}_t$ such that  
$$\mathcal{U}_t = \left\{\bm{u}\in\mathcal{S}^{n-1}: \bm{V}_t \bm{u} \le \bm{w}_t\right\}.$$
\end{lemma}

In the deterministic setting of the revealed preference problem, all the $\bm{u}_t$'s are identical and the learning problem is thus reduced to finding one feasible $\bm{u}$ in the set of $\cap_{t=1}^T \mathcal{U}_t$. But in a stochastic setting, it may happen that the set of $\cap_{t=1}^T \mathcal{U}_t$ is empty. Figure \ref{fig:vis} provides a conceptual visualization of this challenge of ``empty intersection''. Each blue solid segment denotes one such $\mathcal{U}_t$ and the blue dashed line represents a value of $\bm{u}$ that appears most frequently in these $\mathcal{U}_t$'s. We remark that the figure is just for illustrative purpose as the problem may not be well-defined in the 1-dimensional case. 

From an estimation viewpoint, the goal is to estimate the distribution of $\mathcal{P}_u$ without the knowledge of the realized samples $\bm{u}_t$'s, but merely with the knowledge of $\mathcal{U}_t$ to which $\bm{u}_t$ belongs. The sample efficiency of the estimation procedure is naturally contingent on the dispersion of $\mathcal{U}_t$ which is essentially determined by the generation of $(\bm{a}_t, b_t).$ For example, if all the $\mathcal{U}_t$'s coincide with each other, then one can hardly learn much about the underlying $\mathcal{P}_u.$ In this paper, we aim to pinpoint conditions for $\mathcal{P}_{\bm{a},b}$ such that the learning of $\mathcal{P}_u$ is possible. Also, an alternative way to measure the estimation accuracy is to evaluate the predictive performance of the estimated model on new observations generated from $\mathcal{P}_{\bm{a},b},$ and such performance bounds generally bear less dependency on the distribution of  $\mathcal{P}_{\bm{a},b}.$ We also provide theoretical guarantees in this sense. 

\section{Gaussian Setting}

\label{sec_gaussian}

In this section, we consider a setting where the distribution $\mathcal{P}_u$ follows the von Mises-Fisher distribution parameterized by $\bm{\theta}=(\bm{\mu},\kappa)$ with the density function
$$f(\bm{u};\bm{\theta}) \coloneqq \frac{\exp\left(-\kappa \bm{\mu}^\top \bm{u}\right)}{ \int_{\bm{u}\in \mathcal{S}^{n-1}}\exp(-\kappa \bm{\mu}^\top \bm{u}) d\bm{u}}\propto \exp\left(-\kappa \bm{\mu}^\top \bm{u}\right).$$
Here the vector $\bm{\mu}\in\mathbb{R}^n$ represents the mean direction and the parameter $\kappa>0$ controls the concentration of the distribution. The deterministic setting of the revealed preference problem can be viewed as the case when $\kappa=\infty$ and then the distribution degenerates to a point-mass distribution on the unit sphere. Denote the true parameters of the distribution $\mathcal{P}_u$ by $\bm{\theta}^* = (\bm \mu^*, \kappa^*).$ Then the likelihood of the dataset $\mathcal{D}_t$ under a parameter $\bm{\theta}$ is 
$$\prob(\mathcal{D}_t|\bm{\theta}) \coloneqq \prod_{t=1}^T \prob\left((\bm{x}^*_t,\bm{a}_t,b_t)|\bm{\theta}\right) = \prod_{t=1}^T \int_{\bm{u}\in \mathcal{U}_t} f(\bm{u};\bm{\theta}) du.$$
We remark that the maximum likelihood approach cannot be applied here for two reasons. First, the integration of $f(\bm{u};\bm{\theta})$ over the region $\mathcal{U}_t$ is not closed-form. The first point is not only pertaining to the Gaussian parameterization of $\mathcal{P}_u$. The scale-invariant property of the utility vector restricts the distribution $\mathcal{P}_u$ to a unit sphere or a simplex, and consequently, the likelihood function inevitably involves the non-closed-form integration. This issue can be partially resolved by using the Monte Carlo method to approximate the integration, and a good thing is that the same integrand is shared across all the observations. Second, the likelihood function is not analytical in $\bm{\theta}$. Thus this prevents the usage of gradient-based algorithms to solve the problem and also makes it difficult to derive theoretical guarantees for the maximum likelihood estimator.

We propose a Bayesian perspective for the problem: instead of identifying the parameter that maximizes the likelihood function, we directly draw samples from the posterior distribution. We will see shortly that the approach can be justified through a concentration property of the posterior distribution. Suppose we have a prior distribution $\prob_0(\bm{\theta})$ and then we can define the posterior distribution by
\begin{align*}
\prob_T(\bm\theta) & \coloneqq
\frac{\prob_0(\bm\theta)\cdot \prob\left(\mathcal{D}_t|\bm\theta\right)}{\prob\left(\mathcal{D}_t\right)} \\
& \propto \prob_0(\bm\theta)\cdot \prod_{t=1}^T \int_{\bm{u}\in \mathcal{U}_t} f(\bm{u};\bm{\theta}) du.
\end{align*}
With slight abuse of notation, we use $\prob_T(\cdot)$ (or $\prob_0(\cdot)$) to refer to both the density function and the probability measure of the posterior (or prior) distribution. We make the following assumption on the prior distribution.

\begin{assumption}
We assume the concentration parameter $\kappa^*\in(\underline{\kappa}, \bar{\kappa})$ where $\underline{\kappa}, \bar{\kappa}$ are two known positive constants. The prior distribution $\prob_{0}(\bm{\theta})$ is a uniform distribution on $ \mathcal{S}^{n-1}\times(\underline{\kappa},\bar{\kappa})$.
\label{assp_prior}
\end{assumption}

\begin{theorem}
\label{thm:post_conv}
Let
$$\Theta_T \coloneqq \left\{\bm{\theta}\in  \mathcal{S}^{n-1}\times(\underline{\kappa},\bar{\kappa}): \mathcal{W}\left(\prob\left((\bm{x}^*_t,\bm{a}_t,b)|\bm{\theta}\right), \prob\left((\bm{x}^*_t,\bm{a}_t,b)|\bm{\theta}^*\right)\right)\le \max\left(8,8\bar{\kappa}\right)\frac{n \cdot \log T}{T^{1/2-\alpha}}\right\}$$
where $\mathcal{W}(\cdot, \cdot)$ is the Wasserstein distance between two distributions supported on $\mathcal{X} \times \mathbb{R}^n_+ \times \mathbb{R}_+$ equipped with Euclidean metric. Then, under Assumptions \ref{assp_distr}-\ref{assp_prior},
$$
   1- \prob_T(\Theta_{T})\rightarrow 0\ \text{in probability as } T\rightarrow \infty,
$$
for any $\alpha \in [0,1/2]$. Specifically, the following inequality holds
$$ \mathbb{E}\left[ \prob_T\left(\Theta_{T}\right)
        \right]
        \geq
        1-\frac{2}{T}-\frac{1}{4T^{2\alpha}\log ^2T}.
$$
where the expectation is taken with respect to the random distribution $\mathbb{P}_T(\cdot)$ (essentially, with respect to the dataset $\mathcal{D}_T.$)
\end{theorem}

Theorem \ref{thm:post_conv} justifies the approach of posterior sampling. We first remark that the Bayesian sampling approach is just proposed to estimate the parameters, but all the theoretical results are stated in frequentist language. The proof of Theorem \ref{thm:post_conv} follows the standard analysis of the convergence of the posterior distribution \citep{ghosal2000convergence,chae2021posterior}. While similar results should also hold for other underlying distribution of $\mathcal{P}_u$, the von Mises-Fisher distribution provides much analytical convenience in deriving the bound. Each $\bm{\theta}$, together with the distribution of $\mathcal{P}_{\bm{a},b}$, defines a distribution over the space of $(\bm{x}^*_t,\bm{a}_t,b)$. As we use observations $(\bm{x}^*_t,\bm{a}_t,b)$'s to identify the true $\bm{\theta}^*$, the set $\Theta_T$ defines a set of indistinguishable $\bm{\theta}$'s based on the Wasserstein distance between distributions of $(\bm{x}^*_t,\bm{a}_t,b)$. The set $\Theta_T$ shrinks as $T\rightarrow \infty.$ The posterior sampling approach samples from the distribution $\prob_T(\cdot),$ and Theorem \ref{thm:post_conv} states that the samples will be concentrated in set $\Theta_T$ with high probability. The posterior distribution $\prob_T(\cdot)$ is dependent on the dataset $\mathcal{D}_T,$ so it is a random distribution itself and the results in Theorem \ref{thm:post_conv} are stated in either convergence in probability or expectation. As a side note, the Wasserstein distance in the theorem is not critical and it can be replaced with other distances such as the total variation distance and the Hellinger distance.

Intuitively, Theorem \ref{thm:post_conv} says that for some $\bm{\theta}$ such that the likelihood distribution $\prob\left((\bm{x}^*_t,\bm{a}_t,b)|\bm{\theta}\right)$ differs from $ \prob\left((\bm{x}^*_t,\bm{a}_t,b)|\bm{\theta}^*\right)$ to a certain extent, the posterior $\prob_T(\cdot)$ is unlikely to generate such $\bm{\theta}.$ In other words, the posterior distribution identifies the true $\bm{\theta}^*$ up to some ``equivalence'' in the likelihood distribution space. The following corollary formalizes this intuition that if there is an equivalence between the likelihood distribution space and the underlying parameter space, then the posterior distribution is capable of identifying the true parameter.

\begin{corollary}
\label{coro:Stoch}
Suppose 
$$\mathcal{W}\left(\prob\left((\bm{x}^*_t,\bm{a}_t,b)|\bm{\theta}\right),\prob\left((\bm{x}^*_t,\bm{a}_t,b)|\bm{\theta}^*\right)\right)>0$$
for all $\bm{\theta}\neq \bm{\theta}^*\in\mathcal{S}^{n-1}\times [\underline{\kappa}, \bar{\kappa}]$. Then the posterior distribution $\mathbb{P}_T(\cdot)$ will converge to the point-mass distribution on $\bm{\theta}^*$ almost surely as $T\rightarrow \infty$.
Moreover, suppose there exists a constant $L>0$ satisfying
\begin{equation}
\mathcal{W}\left(\prob\left((\bm{x}^*_t,\bm{a}_t,b)|\bm{\theta}\right), \prob\left((\bm{x}^*_t,\bm{a}_t,b)|\bm{\theta}^*\right)\right)
\geq L\cdot \|\bm{\theta}- \bm{\theta}^*\|_2,
\label{eqn:equiv}
\end{equation}
for all $\bm{\theta}\neq \bm{\theta}^*\in\mathcal{S}^{n-1}\times [\underline{\kappa}, \bar{\kappa}]$. 

Under Assumptions \ref{assp_distr}-\ref{assp_prior}, the following inequality holds with probability no less than $1-\frac{1}{T^{2\alpha}\log^2 T},$
\begin{align*}
        \mathbb{E}_T\left[
            \|\bm{\theta}_T-\bm{\theta}^*\|_2
        \right]
        \leq \max\left(9,9\bar{\kappa}\right)\cdot\frac{n \cdot \log T}{L \cdot T^{1/2-\alpha}}
    \end{align*}
where $\bm{\theta}_T$ is sampled from the posterior distribution $\mathbb{P}_T(\cdot).$
\end{corollary}

The corollary states that when there is some equivalence between the likelihood distribution space and the parameter space as \eqref{eqn:equiv}, the true parameter is identifiable. The first part of the corollary states a consistency result that as long as all the $\bm{\theta}\neq \bm{\theta}^*$ are distinguishable from $\bm{\theta}^*$ through the likelihood function, then the posterior sampling will eventually identify the true $\bm{\theta}^*.$ The second part relates to the convergence rate with an equivalence parameter $L$. 

In Assumption \ref{assp_distr}, we assume the constraint input $(\bm{a}_t, b_t)$ is generated from some distribution $\mathcal{P}_{\bm{a},b}.$ We note that Theorem \ref{thm:post_conv} and Corollary \ref{coro:Stoch} hold without any additional assumption on $\mathcal{P}_{\bm{a},b}$, but the space topology of the likelihood distribution is highly dependent on $\mathcal{P}_{\bm{a},b}.$ Specifically, a different distribution of $(\bm{a}_t, b_t)$ determines the separateness of the parameter space through affecting the value of $L$ in \eqref{eqn:equiv} or even its existence. The value $L$ of a specific distribution of $\mathcal{P}_{\bm{a},b}$ can be examined through simulation. So if the learner has some flexibility in choosing the distribution of $\mathcal{P}_{\bm{a},b}$, the optimal choice would be the one that corresponds to a larger value of $L.$ If the constraint input $(\bm{a}_t, b_t)$ is not randomly generated but can be actively chosen as the query-based preference learning problem, the results in Theorem \ref{thm:post_conv} and Corollary \ref{coro:Stoch} still hold by conditioning on all the $(\bm{a}_t, b_t)$'s. Unlike the deterministic case where the utility vector $\bm{u}$ is fixed for all the observations, the stochastic nature of the problem setup here makes it generally very complicated to fully extract the benefit of designing $(\bm{a}_t, b_t)$'s by the learner. We leave it as a future open question.

\begin{corollary}
\label{predict_coro_gaussian}
Suppose $\mathcal{P}_{\bm{a},b}$ is a discrete distribution with a finite support. Let $(\bm{a}, b)$ be a new sample from $\mathcal{P}_{\bm{a},b}$, i.e., independent from the dataset $\mathcal{D}_T$, and let $\bm{\theta}_T=(\bm{\mu}_T, \kappa_T)$ be a sample from the posterior distribution $\mathbb{P}_T(\cdot).$ Denote $\tilde{\bm{x}}^*$ and $\bm{x}^*$ as the optimal solutions of LP$(\bm{\mu}_T, \bm{a}, b)$ and LP$(\bm{\mu}^*, \bm{a}, b)$, respectively. Then, under Assumptions \ref{assp_distr}-\ref{assp_prior}, the following inequality holds  with probability no less than $1-\frac{1}{T^{2\alpha}\log^2 T}$,
    $$
    \mathbb{E}\left[\|\tilde{\bm{x}}^*-\bm{x}^*\|_2\right]
    \leq
    \max\left(16,16\bar{\kappa}\right)\frac{n \cdot \log T}{T^{1/2-\alpha}},
    $$
where the expectation is taken with respect to both the posterior distribution $\mathbb{P}_T(\cdot)$ and $(\bm{a},b).$
\end{corollary}

Corollary \ref{predict_coro_gaussian} provides an upper bound on the predictive performance of the posterior distribution. Specifically, we want to predict the optimal solution of a linear program specified by $\bm{\mu}^*$ (proportionally to $\E[\bm{u}]$) and a new sample of the constraint $(\bm{a}, b)$, and the prediction $\tilde{\bm{x}}^*$ is based on a posterior sample. We know from Theorem \ref{thm:post_conv} that the posterior distribution concentrates on those $\bm{\theta}$'s that are indistinguishable from the true $\bm{\theta}^*$ in terms of the likelihood. Speaking of the predictive performance, we only concern the distribution of the optimal solution (equivalently, the likelihood), but do not require the identification of exact true $\bm{\theta}^*,$ so Corollary \ref{predict_coro_gaussian} does not require the condition \eqref{eqn:equiv} to hold. Intuitively, the prediction of the optimal solution on a new observation $(\bm{a},b)$ can be viewed as a condition distribution of the optimal solution given $(\bm{a}, b)$. While the definition of $\Theta_T$ in Theorem \ref{thm:post_conv} concerns the joint distribution of the optimal solution and $(\bm{a},b)$, the finite-support condition on $\mathcal{P}_{\bm{a},b}$ in Corollary \ref{predict_coro_gaussian} transforms the result on the joint distribution to the conditional distribution. 

\section{$\delta$-Corruption Case}

\label{sec_delta}

In this section, we consider a setting where the utility vector is specified by
\begin{equation}
 \bm{u}_t
=\begin{cases}
 \bm{u}^*, &\text{ w.p. } 1-\delta, \\
 \mathcal{P}_u',  & \text{ w.p. } \delta,
\end{cases}
\label{u_generate}
\end{equation}
where $\bm{u}^*\in\mathcal{S}^{n-1}$ is a fixed vector, $\delta\in[0,1]$, and $\mathcal{P}_u'$ is an arbitrary distribution that corrupts the inference of $\bm{u}^*$. The deterministic setting of the revealed preference problem in literature can be viewed as the case of $\delta=0$, and the Gaussian setting in the previous section can be viewed as the case of $\delta=1$ and $\mathcal{P}_u'$ being the von-Mises Fisher distribution. In this setting, we do not aim to learn the distribution of $\mathcal{P}_u'$, but rather our goal is to identify the vector $\bm{u}^*$ using the dataset $\mathcal{D}_T.$

A natural idea to estimate $\bm{u}^*$ is by solving the following optimization problem:
$$\text{OPT}_{\delta} \coloneqq \max_{\bm{u}\in\mathcal{S}^{n-1}}\  \sum_{t=1}^T I_{\mathcal{U}_t}(\bm{u})$$
where the indicator function $I_{\mathcal{E}}(e)=1$ if $e\in\mathcal{E}$ and $I_{\mathcal{E}}(e)=0$ otherwise. The rationale for the optimization problem is that for the $t$-th observation, a vector $\bm{u}$ is consistent with the observation, i.e., $\bm{x}_t^*$ is the optimal solution of LP$(\bm{u},\bm{a}_t, b_t)$, if and only if $I_{\mathcal{U}_t}(\bm{u})=1$. Thus the optimization problem finds a vector $\bm{u}$ that is consistent with the maximal number of observations. The objective function is discontinuous in $\bm{u}$, so we propose the simulated annealing algorithm -- Algorithm \ref{alg:SA} to solve for its optimal solution. 

We first build some connection between the optimization problem OPT$_\delta$ and that of the deterministic setting with $\delta=0$. Let $\bar{\bm{x}}_t^*$ be the optimal solution of LP$(\bm{u}^*, \bm{a}_t, b_t)$ and define 
$$\bar{\mathcal{U}}_t\coloneqq\left\{\bm{u}\in \mathcal{S}^{n-1}: \bar{\bm{x}}_t^* \text{ is an optimal solution of } \text{LP}(\bm{u}, \bm{a}_t, b_t)\right\}.$$
Then the deterministic setting of the revealed preference problem solves
$$\bar{\text{OPT}}_{\delta} \coloneqq \max_{\bm{u}\in\mathcal{S}^{n-1}}\  \sum_{t=1}^T I_{\bar{\mathcal{U}}_t}(\bm{u}).$$
By the setup of the problem, $\bm{u}^*$ is an optimizer of $\bar{\text{OPT}}_{\delta}$ and the optimal objective value is $T.$ The following proposition establishes that the optimization problem OPT$_\delta$ is a contaminated version of $\bar{\text{OPT}}_{\delta}$ and the effect that the contamination has on the objective function can be bounded using $\delta$.

\begin{proposition}
\label{prop:sim}
Under Assumption \ref{assp_distr}, the following inequality holds
\begin{align*}
\prob\left(\max_{\bm{u}\in\mathcal{S}^{n-1}} \left\vert\frac{1}{T}\sum\limits_{t=1}^{T} I_{\mathcal{U}_t}(\bm{u})-\frac{1}{T}\sum\limits_{t=1}^{T}I_{\bar{\mathcal{U}}_t}(\bm{u})\right\vert\leq \delta+\frac{\log T}{\sqrt{T}}\right)
\leq\frac{1}{T}.
\end{align*}
\end{proposition}

When the constraints $(\bm{a}_t,b_t)$'s are generated from some distribution $\mathcal{P}_{\bm{a},b},$ it can happen that there exist some vectors $\bm{u}'$ that are indistinguishable from $\bm{u}^*$ based on the observations $\mathcal{D}_t$ as in the previous Gaussian case. So, we do not hope for an exact recovery of $\bm{u}^*$, but alternatively, we aim to derive a generalization bound for our estimator $\hat{\bm{u}}$. Specifically, we define and analyze the accuracy
$$\text{Acc}(\hat{\bm{u}}) \coloneqq \E\left[I_{\mathcal{U}}(\hat{\bm{u}})\right]\text{ with }\ \mathcal{U}\coloneqq\left\{\bm{u}\in \mathcal{S}^{n-1}: {\bm{x}}^* \text{ is an optimal solution of } \text{LP}(\bm{u}, {\bm{a}}, {b})\right\}$$
where $\hat{\bm{u}}$ is our estimator of $\bm{u}^*$, $({\bm{a}}, {b})$ is a new sample from the distribution $\mathcal{P}_{\bm{a},b}$, ${\bm{u}}$ is a new sample following the law of \eqref{u_generate}, and ${\bm{x}}^*$ is the optimal solution of LP$({\bm u}, {\bm a}, {b})$. In other words, the quantity captures the probability that $\hat{\bm{u}}$ is consistent with a new (unseen) observation, and we know that for the true parameter, $\text{Acc}(\bm{u}^*)\ge 1-\delta$, which serves as a performance benchmark. 

The challenge for deriving a bound on $\text{Acc}(\hat{\bm{u}})$ arises from the discontinuity of the objective function OPT$_\delta$. The existing methods for deriving generalization bound largely rely on the continuity and the Lipschitzness of the loss function. To make it worse, from Lemma \ref{lemma_linear}, we know that $\mathcal{U}_t$ is specified by $(\bm{V}_t, \bm{w}_t)$ and the $\bm{V}_t$'s are of different dimensions for different $t$'s. To overcome these challenges, we devise the following $\gamma$-margin objective function. Specifically, we first define a parameterized version of $\mathcal{U}_t$ by
$$\mathcal{U}_t(\gamma) \coloneqq \left\{\bm{u}\in\mathcal{S}^{n-1}: \bm{V}_t \bm{u} \le \bm{w}_t - \gamma \bm{e} \right\}$$
where $\gamma$ is a positive constant and $\bm{e}$ is an all-one vector. It is obvious that $\mathcal{U}_t(\gamma)\subset \mathcal{U}_t.$ Accordingly, we define the $\gamma$-margin optimization problem by
$$\text{OPT}_{\delta}(\gamma) \coloneqq \max_{\bm{u}\in\mathcal{S}^{n-1}}\  \sum_{t=1}^T I_{\mathcal{U}_t(\gamma)}(\bm{u}).$$

\begin{proposition}
\label{prop_generalize}
 Under Assumption \ref{assp_distr}, the following inequality holds with probability no less than $1-\epsilon$,
    $$
      \max_{\bm{u}\in \mathcal{S}^{n-1}} \ 
        \frac{1}{T}\sum\limits_{t=1}^{T}I_{\mathcal{U}_{t}(\gamma)}(\bm{u})- \mathrm{Acc}(\bm{u})
        \le 
        4\sqrt{\frac{\log(T)}{{\underline{a}^2\gamma^2 T}}}+
        6\sqrt{\frac{\log(T/\epsilon)}{{T}}}        
    $$
    for $\epsilon\in(0,1).$
\end{proposition}

Proposition \ref{prop_generalize} relates the generalization accuracy of any arbitrary $\bm{u}$ with the corresponding objective value of the $\gamma$-margin optimization problem. As $\gamma$ increases, the objective function will decrease, so the right-hand-side becomes tighter. Importantly, the accuracy is defined by the original indicator function (or equivalently, $\mathcal{U}_t$), while the objective value is defined by the $\gamma$-margin indicator function (or equivalently, $\mathcal{U}_t(\gamma)$). The implication is that when we optimize the $\gamma$-margin objective, we can still obtain a bound on the original accuracy Acc$(\bm{u})$ for sufficiently large $\gamma$.

\begin{theorem}
\label{thm_generalize}
Suppose $\mathcal{P}_{\bm{a},b}$ is a continuous distribution and it has a density function upper bounded by $\bar{p}$. Then, under Assumption \ref{assp_distr}, the following inequality holds 
    \begin{align*}
       \prob\left(\mathrm{Acc}(\hat{\bm{u}})
        \ge
        1- \delta
        -
       \frac{\bar{p}}{\min\limits_{i:u_i^*\not=0}|u_i^*|\cdot T^{1/4}}
        -
        \frac{20n^2\log(T)}{\underline{a}T^{1/4}}\right) \ge 1-\frac{4}{T},
    \end{align*}
where $\hat{\bm{u}}$ is one optimal solution of $\text{OPT}_{\delta}(\gamma)$ with $\gamma=\frac{1}{4n^2T^{1/4}}$.
\end{theorem}

Theorem \ref{thm_generalize} states a generalization bound on the accuracy of $\hat{\bm{u}}$ for continuous distributions of $(\bm{a},b).$ From Proposition \ref{prop_generalize}, a larger $\gamma$ leads to a smaller gap between the accuracy and the $\gamma$-margin objective function. Meanwhile, a smaller $\gamma$ leads to a smaller gap between the optimal objective value OPT$_\delta(\gamma)$ and $1-\delta$. In the extreme case of $\gamma=0$, $\E[\text{OPT}_\delta(0)]=\text{Acc}(\bm{u}^*)\ge1-\delta.$ Theorem \ref{thm_generalize} optimizes the value of $\gamma$ to trade off these two aspects. We note that the continuous distribution and the upper bound on the density function make it possible to bound the gap of the second aspect. Finally, we remark that the design of $\gamma$-margin loss function is inspired from the max-margin classifier, but the analysis is entirely different. For the max-margin classifier, the introduction of the margin aims to make the underlying loss function 1-Lipschitz so that a generalization bound using Rademacher complexity can be derived. But for here, our $\gamma$-margin objective function is still a discontinuous one.

\section{Computational Aspects and Discussions}

\begin{algorithm}[ht!]
\caption{Posterior Sampling for the Gaussian Setting}
\label{alg_MCMC}
\begin{algorithmic}[1] 
\State Input: dataset $\mathcal{D}_T=\{(\bm{x}_t^*,\bm{a}_t,b_t)\}_{t=1}^{T}$, number of iterations $K$
\State Initialize $\bm{\theta}^{(0)}$ by randomly sampling from the prior distribution  $\prob_{0}(\bm{\theta})$
\For {$k=1,...,K$}
    \State Draw a random $\bm{\theta}'$ from a pre-determined proposal distribution $\mathbb{Q}(\bm{\theta}'|\bm{\theta}^{(k-1)})$
    \State Compute the acceptance rate: $$r=\min\left\{\frac{\prob_T(\bm\theta') }{\prob_T(\bm\theta^{(k-1)})},1\right\}$$
    \State Set
    \begin{align*}
        \bm{\theta}^{(k)}=
        \begin{cases}
           \bm{\theta}',&\text{w.p. $r$}\\
            \bm{\theta}^{(k-1)}, &\text{w.p. $1-r$}
        \end{cases}
    \end{align*}
    \State 
\EndFor
\State Output: $\bm{\theta}^{(K)}$ 
\end{algorithmic}
\end{algorithm}

In the previous section, we developed theoretical results for both Gaussian and $\delta$-corruption settings. Now we discuss computational aspects with respect to the sampling of the posterior $\mathbb{P}_T(\cdot)$ and the optimization of OPT$_\delta(\gamma).$ As mentioned earlier, the posterior sampling removes the complication of optimizing over $\mathbb{\theta}$ in the maximum likelihood estimation, but still inevitably needs to deal with the sampling and numeric approximation of the likelihood function. Algorithm \ref{alg_MCMC} describes a standard Metropolis–Hastings algorithm to sample from the posterior distribution $\mathbb{P}_{T}(\cdot).$ In the numerical experiments, we choose the proposal distribution $\mathbb{Q}$ to be a Gaussian random perturbation, i.e., $\bm{\theta}' = \text{Proj}(\bm{\theta}^{(k-1)}+\bm{\epsilon})$ where $\bm{\epsilon}$ follows a Gaussian distribution and the projection ensures that $\bm{\theta}'$ stays on the sphere $\mathcal{S}^{n-1}.$ For the acceptance ratio, as the posterior distribution is not in closed form, a Monte Carlo subroutine is needed to estimate the ratio.

\begin{algorithm}[ht!]
\caption{Simulated annealing algorithm for $\delta$-corruption}
\label{alg:SA}
\begin{algorithmic}[1] 
\State Input: dataset $\mathcal{D}_T=\{(\bm{x}_t^*),\bm{a}_t,b_t\}_{t=1}^{T}$, margin $\gamma$, number of iterations $K$, interval length $\tau$
\State Initialize an initial (temperature) $\eta>0$ and the reduction rate $c\in(0,1)$ 
\State Randomly generate the first estimate $\bm{u}^{(0)}$
\For {$k=1,...,K$}
    \If{$k \text{ mod } \tau = 0$}
        \State Update 
        $\eta \leftarrow c \cdot \eta $
    \EndIf
    \State Draw a proposal $\bm{u}'$ from a predetermined proposal distribution $\mathbb{Q}(\bm{u}'|\bm{u}^{(k-1)})$
    \State Compute the acceptance rate:
    \begin{equation}
r=\min\left\{\exp\left\{
        \frac{1}{\eta}\cdot\left(\sum_{t=1}^T I_{\mathcal{U}_t(\gamma)}(\bm{u}')-\sum_{t=1}^T I_{\mathcal{U}_t(\gamma)}(\bm{u}^{(k-1)})\right)      
    \right\},1\right\}
    \label{algo2_likeli}
    \end{equation}
    \State Set
    \begin{align*}
        \bm{u}^{(k)}=
        \begin{cases}
            \bm{u}',&\text{ w.p. $r$}\\
            \bm{u}^{(k-1)}, &\text{ w.p. $1-r$}
        \end{cases}
    \end{align*}
\EndFor
\State Output: $\bm{u}^{(K)}$ 
\end{algorithmic}
\end{algorithm}

Algorithm \ref{alg:SA} presents a simulated annealing algorithm to solve the optimization problem OPT$_{\delta}(\gamma)$ in Section \ref{sec_delta}. It takes a similar MCMC routine as Algorithm \ref{alg_MCMC} and we use the same Gaussian random perturbation for the proposal distribution $\mathbb{Q}$. As the temperature parameter $\eta$ decreases, the sampling distribution in Algorithm \ref{alg:SA} will gradually be more concentrated on the optimal solution set of OPT$_{\delta}(\gamma)$. Algorithm \ref{alg:SA} can be implemented more efficiently than Algorithm \ref{alg_MCMC} in that the likelihood ratio calculation in \eqref{algo2_likeli} is analytical.

Table \ref{tab:acc} reports some numerical results for the two algorithms. For both the Gaussian and $\delta$-corruption settings, we consider three distributions of $\mathcal{P}_{\bm{a}, b}$: (i) a uniform distribution where $\bm{a}\sim$Unif$([1,2]^n)$ and $b\sim$Unif$([1,n])$; (ii) a discrete distribution where $\text{Unif}(\{1,2\}^{n})$ and $b\sim$Unif$({1,...,n})$; (iii) a fixed-$\bm{a}$ distribution where $\bm{a}=(1,...,1)^\top$ and $b\sim$Unif$({1,...,n})$. For the Gaussian case, the true parameters $(\bm{\mu}^*,\kappa^*)$ are uniformly generated from $\mathcal{S}^{n-1}\times[1,10]$, and the accuracy is calculated by $(\bm{\mu}^* \bm{x}^* - \bm{\mu}^* \tilde{\bm{x}}^*)/\bm{\mu}^* \bm{x}^*$ where $\bm{x}^*$ and $\tilde{\bm{x}^*}$ are defined in Corollary \ref{predict_coro_gaussian}. For the Gaussian case, $\bm{u}^*$ is uniformly generated from $\mathcal{S}^{n-1}$ and $\delta$ is set to be $0.1$, and the accuracy is calculated by Acc$(\hat{\bm{u}})/$Acc$(\bm{u}^*)$ where Acc$(u)$ is defined in Section \ref{sec_delta}. The numbers in Table \ref{tab:acc} are reported based on an average of 20 simulation trials, and we run both Algorithm \ref{alg_MCMC} and Algorithm \ref{alg:SA} for $K=1000$ iterations.

We make the following observations from the numerical experiments. First, we remark that the theoretical results in the previous sections provide strong guarantees on the convergence property of the posterior distribution. So the deterioration of the algorithm performance for the case when $n=25$ is solely caused by the inaccuracy of the approximate sampling in either Algorithm \ref{alg_MCMC} or Algorithm \ref{alg:SA}. Such inaccuracy can definitely be mitigated to some extent by a more efficient algorithm implementation such as parallel computing. However, we argue that the performance deterioration as $n$ grows may point to a curse of dimensionality that is intrinsic to this estimation problem. Essentially, we aim to estimate a high-dimensional distribution only through partial information, i.e., the sets $\mathcal{U}_t$'s. On the positive end, the algorithms work well for $n\le 10,$ so if the learner has the power of choosing $(\bm{a}_t,b_t)$, s/he can break up the high-dimensional estimation problem into a number of low-dimensional estimation problems by focusing on a handful of dimensions each time. Moreover, we provide a visualization of the condition \eqref{eqn:equiv} in Figure \ref{fig:gaurnd} for $n=5$ calculated based on simulation. The visualization supports the existence of $L$ and thus the identifiability of the true parameters when the posterior sampling can be accurately fulfilled. 
0
\begin{figure}[h]
\begin{minipage}[T]{0.63\textwidth}
\begin{tabular}{ cc c c c c}
\toprule
 & & {$n=3$} & {$n=5$} & {$n=10$} & {$n=25$}\\
\midrule
 & {(i)}& 99.9\% & 99.9\% & 98.8\% & 59.9\%\\\
{ Gaussian}  & {(ii)} & 99.9\% & 99.3\% &  96.9\% & 56.9\%\\
                    & {(iii)} & 99.9\% & 94.8\%&  92.6\% & 68.5\%\\
\midrule
  & {(i)} & 99.9\% & 96.7\% &   96.0\% & 55.7\% \\
 { $\delta$-corru.}& {(ii)} & 99.6\% &  97.9\%  &  97.6\%&  63.1\%\\
& {(iii)} & 99.9\% & 98.7\%  &  87.1\%& 58.7\%\\
\bottomrule
\end{tabular} 
\captionof{table}{Predictive accuracies under two settings.}
\label{tab:acc}
\end{minipage}%
  \begin{minipage}[T]{.5\textwidth}
  \includegraphics[width=\textwidth]{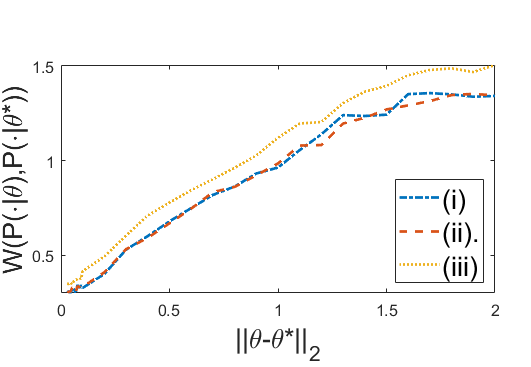}
  \caption{Visualization of \eqref{eqn:equiv}.}
  \label{fig:gaurnd}
\end{minipage}%
\end{figure}

We conclude our discussion with the following remarks.
 
Query-based model with learner-chosen $(\bm{a}_t,b_t)$: In this paper, we have focused on the case where the constraints $(\bm{a}_t,b_t)$'s are stochastically generated. When the concentration parameter $\kappa$ is known for the Gaussian case, there is an efficient way of learning $\bm{\mu}$ through choosing $(\bm{a}_t, b_t)$'s (See the Appendix). In addition, the numerical experiments above also inspire a method that dismantles the high-dimensional estimation problem into a number of low-dimensional problems. Another interesting and important question is whether there exist designs of $(\bm{a}_t,b_t)$'s such that the posterior sampling can be more efficiently carried out. 

Multiple constraints and nonlinear utility: The results in this paper are presented under the setting of a linear objective and a single constraint. We emphasize that the results can be easily generalized to the case of multiple constraints and parameterized nonlinear utility. Thus our result can be viewed as a preliminary effort to address the problem of stochastic inverse optimization. Our conjecture is that when the set $\mathcal{U}_t$ corresponds to a multiple-constraint problem, it may feature more structure and thus facilitate the learning of the utility distribution.

Choice modeling: The stochastic utility model in our paper also draws an interesting connection with the literature on choice modeling, which is a pillar for the pricing and assortment problems in revenue management \citep{talluri2004theory, gallego2019revenue}. For most of the existing choice models, the learning problem can be viewed as a special case of our study by letting $\bm{a}=(1,...,1)^\top$ and $b=1$. The results in our paper complement to this line of literature in developing a model where customers can make multiple purchases. 

\bibliographystyle{informs2014} 
\bibliography{sample.bib} 

\appendix
\section{Auxiliary Lemmas}
We present some preliminary lemmas in this section. Most of them are basic inequalities in Information Theory, so they are only for auxiliary purposes in our proofs of the main theorems and the corollaries. 

\begin{lemma}[Pinsker’s inequality]
\label{lem:pinsker}
    For any two distributions $\mathcal{P}_1$ and $\mathcal{P}_2$, 
    $$
        D_{TV}(\mathcal{P}_1,\mathcal{P}_2)
        \leq
        \sqrt{\frac{1}{2}D_{KL}(\mathcal{P}_1,\mathcal{P}_2)},
    $$
    where $D_{TV}(\cdot,\cdot)$ denotes the total variation distance between two distributions, and $D_{KL}(\cdot,\cdot)$ denotes the KL-divergence between distributions. 
\end{lemma}
\begin{proof}
    We refer to Lemma 6.2 of the book \cite{gray2011entropy}.
\end{proof}

\begin{lemma}[Data processing inequality]
    \label{lem:kldes}    
    Let $\{K_{\lambda}\}_{\lambda\in\mathcal{X}}$ be a set of random variables indexed by parameter $\lambda$ in some space $\mathcal{X}$. Consider two random variables $\Lambda_1,\Lambda_2$ taking values in $\mathcal{X}$. The following inequality holds 
    $$
        D_{KL}(K_{\Lambda_1},K_{\Lambda_2})\leq D_{KL}(\Lambda_1,\Lambda_2),
    $$
    where $D_{KL}(\cdot,\cdot)$ denotes the KL-divergence between two distributions. Here $\{K_\lambda\}_{\lambda\in\mathcal{X}}$ is also called a Markov kernel, a transition probability distribution, or a statistical kernel.
\end{lemma}
\begin{proof}
    We refer to Theorem 14 of the article \cite{liese2006divergences}.
\end{proof}

\begin{lemma}[Packing number]
    \label{lem:pacnum}
    Let $\mathcal{B}_r$ denote the ball in $\mathbb{R}^n$  centered at original point with radius $r$. Then, the $\epsilon$-packing number of $\mathcal{B}_r$ is bounded by 
    $$
    \left(1+\frac{2r}{\epsilon}\right)^n.
    $$
    In other words, there exist at most $\left(1+\frac{2r}{\epsilon}\right)^n$ disjoint $\frac{\epsilon}{2}$ balls in $\mathcal{B}_r$.
\end{lemma}
\begin{proof}
    Assume that there are $M$ disjoint $\epsilon/2$-balls. Then, the total volume of those $M$ balls cannot be larger than the volume of a $r+\epsilon/2$ ball, i.e.,
    \begin{align*}
        M\cdot\left(\frac{\epsilon}{2}\right)^n
        \leq
        (r+\frac{\epsilon}{2})^n,
    \end{align*}
    which implies 
    $$
        M\leq\left(1+\frac{2r}{\epsilon}\right)^n.
    $$
\end{proof}

\begin{lemma}[Hoeffding's inequality]
Let $X_1, ..., X_T$ be independent random variables such that $X_t$ takes its values in $[u_t, v_t]$ almost surely for all $t\le T.$
Then for every $s>0,$
$$\prob\left(\left\vert \frac{1}{T}\sum_{t=1}^n X_t -\E X_t \right\vert \ge s\right) \le 2\exp\left(-\frac{2T^2s^2}{\sum_{i=1}^{n}(u_t-v_t)^2}\right).$$
\label{HF1}
\end{lemma}
\begin{proof}
    We refer to Chapter 2 of the book \citep{boucheron2013concentration}.
\end{proof}

\begin{lemma}[Doob's consistency theorem]
    \label{lem:doob}
    Suppose that $\mathbb{P}(\cdot|\bm{\theta})\not=\mathbb{P}(\cdot|\bm{\theta}')$ whenever $\bm{\theta}\not=\bm{\theta}'$. Then, for every prior probability measure on the parameter space, the sequence of posterior measures converges to the point mass distribution of the true parameter in distribution for almost every $\bm{\theta}$.
\end{lemma}
\begin{proof}
    We refer to 10.10 of the book \cite{van2000asymptotic}.
\end{proof}

\begin{lemma}
\label{lem:pacbay}
Let $\mathcal{F}$ be a set of functions whose domain is the support of the distribution of $(\bm{x}^*,\bm{a},b)$. For any probability distribution $\tilde{\mathbb{P}}$ on $\mathcal{F}$, the following inequality holds for all $f\in\mathcal{F}$ and all distributions $\tilde{\mathbb{Q}}$ on $\mathcal{F}$ simultaneously
    $$
        \mathbb{E}_{\tilde{\mathbb{Q}}}\left[\mathbb{E}[f(\bm{x}^*,\bm{a},b)]\right]\leq \mathbb{E}_{\tilde{\mathbb{Q}}}\left[\sum\limits_{t=1}^{T}f(\bm{x}^*_t,\bm{a}_t,b_t)\right]+\sqrt{\frac{D_{KL}(\tilde{\mathbb{P}},\tilde{\mathbb{Q}})+\log\frac{T}{\epsilon}+2}{2T-1}}.
    $$
with probability no less than $1-\epsilon$. Here the inner expectation on the left-hand-side is taken with respect to $(\bm{x}^*, \bm{a}, b).$
\end{lemma}
\begin{proof}
    We refer to Theorem 1 of the article \cite{mcallester2003pac}.
\end{proof}

The following lemma provides a useful bound for the modified Bessel function of the first kind. This function is closely related to the density of the von Mises–Fisher distribution. Typically, the modified Bessel function of the first kind is denoted by $I_{\nu}(x)$ with the parameter $\nu$. In this paper, to distinguish between the indicator function and this modified Bessel function, we denote the modified Bessel function by $\tilde{I}_{\nu}(x)$.

\begin{lemma}
    \label{lem:bessel}
    For all $0<x<y$ and $\nu>0$,
    \begin{align*}
        \text{e}^{x-y}\left(\frac{x}{y}\right)^{\nu}
        \leq
        \frac{\tilde{I}_\nu(x)}{\tilde{I}_\nu(y)}
        \leq
        \text{e}^{y-x}\left(\frac{x}{y}\right)^{\nu},
    \end{align*}
    where $\tilde{I}_\nu(\cdot)$ denotes the modified Bessel function of the first kind. 
\end{lemma}
\begin{proof}
    We refer to Chapter 2 of the article \cite{baricz2010bounds}.
\end{proof}

\section{Proof of Section \ref{sec_gaussian}}
In this section, we provide the proofs of Section \ref{sec_gaussian} and analyze the convergence of the posterior distribution. As we mentioned earlier, the Wasserstein distance in Theorems \ref{thm:post_conv} is not essential, and all other equivalent metrics or weaker metrics are also valid, such as the total variation distance, the Hellinger distance and the Prokhorov metric (See \cite{gibbs2002choosing} for more about the relationships between different probability metrics). In this section, we show that Theorem \ref{thm:post_conv} holds with respect to the total variation distance, and then explain how to obtain the result for the Wasserstein distance from there. The main idea is inspired by \cite{ghosal2000convergence}. In this section, we first state and show several key lemmas and then prove the main theorem.

\subsection{Key Lemmas for the Proof of Theorem \ref{thm:post_conv}}

Recall $\mathcal{D}_T=\{(\bm{x}_t^*,\bm{a}_t,b_t)\}_{t=1}^{T}$ is the data set, and $\mathbb{P}((\bm{x}^*,\bm{a},b)|\bm{\theta})$ is the likelihood distribution under the parameter $\bm{\theta}$. In the following, we denote the parameter space $\mathcal{S}^{n-1}\times(\underline{\kappa},\bar{\kappa})$ by $\Theta$, denote the total variation distance between $\mathbb{P}((\bm{x}^*,\bm{a},b)|\bm{\theta}_1)$ and $\mathbb{P}((\bm{x}^*,\bm{a},b)|\bm{\theta}_2)$ by $D_{TV}(\bm{\theta}_1,\bm{\theta}_2)$, where $\bm{\theta}_1,\bm{\theta}_2\in\Theta$ are two parameters, and denote the $\epsilon$-packing number of a parameter subset $\tilde{\Theta}\subset\mathcal{S}\times(\underline{\kappa},\bar{\kappa})$ with a metric $D_{metric}$ by
$
    \mathcal{N}\left(\epsilon,\tilde{\Theta}, D_{metric}\right). 
$ For example, if the underlying metric is the total variation metric, the corresponding $\epsilon$-packing number is denoted by
$
    \mathcal{N}\left(\epsilon,\tilde{\Theta}, D_{TV}\right). 
$

\begin{lemma}[Lemma 7.1 in \citep{ghosal2000convergence}]
\label{lem:testfun}
    Suppose the $\epsilon/2$-packing number of the parameter space $\Theta$ with the total variation distance is bounded by some positive constant $C$, i.e., $$\mathcal{N}\left(\frac{\epsilon}{2},\Theta,D_{TV}\right)\leq C.$$ 
    Then, for $t,j \in \mathbb{N}$, there exist a function $\phi_{t}$, which maps a dataset with $t$ samples $\mathcal{D}_t$ to $[0,1]$, such that
    \begin{align}
        \mathbb{E}_{\bm{\theta}^*}[\phi_t]
        &\leq
        \frac{C\exp(-2t\epsilon^2)}{1-\exp(-2t\epsilon^2)},\label{ieq:tests1}\\
        \sup\limits_{D_{TV}(\bm{\theta},\bm{\theta}^*)>j\epsilon}\mathbb{E}_{\bm{\theta}}[1-\phi_t]
        &\leq
        \exp(-2tj^2\epsilon^2),\label{ieq:tests2}
    \end{align}
where the expectation is taken with respect to the underlying dataset $\mathcal{D}_t$ under a distribution specified by the corresponding parameter.
\end{lemma}
\begin{proof}
    We refer to Lemma 7.1 in \cite{ghosal2000convergence}.
\end{proof}
Intuitively, the lemma states that if the packing number is bounded, we can find a function $\phi_t$ such that its expectation is close to $0$ when taken under a distribution with the true parameter, and it is close to 1 otherwise. In statistics, the function $\phi_t$ naturally serves as a test. In the proof of Theorem \ref{thm:post_conv}, we will see that this test function $\phi_t$ plays an important role in bounding the numerator of the posterior distribution.

To apply this lemma for the proof of Theorem \ref{thm:post_conv}, we first bound the packing number to meet the lemma's condition. Here $D_{KL}(\cdot,\cdot)$ denotes the KL divergence between two distributions.
\begin{lemma}
    \label{lem:klnorm}
    For any two parameters $\bm{\theta}_1=(\bm{\mu}_1,\kappa_1)$ and $\bm{\theta}_2=(\bm{\mu}_2,\kappa_2)$ in $\Theta\coloneqq\mathcal{S}^{n-1}\times(\underline{\kappa},\bar{\kappa})$,
    \begin{align}
        D_{TV}(\bm{\theta}_1,\bm{\theta}_2)
        &\leq
        \sqrt{\frac{1}{2}D_{KL}(\mathbb{P}((\bm{x}^*,\bm{a},b)|\bm{\theta}_1),\mathbb{P}((\bm{x}^*,\bm{a},b)|\bm{\theta}_2))}
        \leq
        \sqrt{\|\bm{\theta}_1-\bm{\theta}_2\|_2}. \label{ieq:klnorm}
    \end{align}
    Furthermore, we have
    \begin{align*}
        \mathcal{N}\left(\epsilon,\Theta,D_{TV}\right)\leq \left(1+\frac{\max(2,2\bar{\kappa})}{\epsilon^2}\right)^n.
    \end{align*}
\end{lemma}
\begin{proof}
    The first inequality  comes directly from Pinsker’s inequality (Lemma \ref{lem:pinsker}). From Lemma \ref{lem:kldes}, we have
    \begin{align*}
   & \ \ \ \ D_{KL}(\mathbb{P}((\bm{x}^*,\bm{a},b)|\bm{\theta}_1),\mathbb{P}((\bm{x}^*,\bm{a},b)|\bm{\theta}_2)) \\ &\le D_{KL}(\mathcal{P}_{\bm{u}}(\cdot|\bm{\theta}_1),\mathcal{P}_{\bm{u}}(\cdot|\bm{\theta}_2))\\
        & 
        =
        \int_{\bm{u}\in\mathcal{S}^{n-1}}
            C_{n}(\kappa_1)\exp(\kappa_1\bm{\mu}_1^\top\bm{u})\cdot\log\left(\frac{C_{n}(\kappa_1)\exp(\kappa_1\bm{\mu}_1^\top\bm{u})}{C_{n}(\kappa_2)\exp(\kappa_2\bm{\mu}_2^\top\bm{u})}\right)\\
        &=
        \int_{\bm{u}\in\mathcal{S}^{n-1}}
            C_{n}(\kappa_1)\exp(\kappa_1\bm{\mu}_1^\top\bm{u})\cdot
            \left(            (\kappa_1\bm{\mu}_1-\kappa_2\bm{\mu}_2)^{\top}\bm{u}+\log\left(\frac{\kappa_1^{n/2-1}\tilde{I}_{n/2-1}(\kappa_2)}{\kappa_2^{n/2-1}\tilde{I}_{n/2-1}(\kappa_1)}\right)
            \right)\\
        &\leq
        \int_{\bm{u}\in\mathcal{S}^{n-1}}
            C_{n}(\kappa_1)\exp(\kappa_1\bm{\mu}_1^\top\bm{u})\cdot
            \left(            (\kappa_1\bm{\mu}_1-\kappa_2\bm{\mu}_2)^{\top}\bm{u}+|\kappa_1-\kappa_2|
            \right)\\
        &\leq
        2\|\kappa_1\bm{\mu}_1-\kappa_2\bm{\mu}_2\|_2
        ,
    \end{align*}
where $$C_n(\kappa)=\left(\int_{\bm{u}\in\mathcal{S}^{n-1}}\exp(\kappa\bm{\mu}^\top\bm{u})\right)^{-1}=\frac{\kappa^{n/2-1}}{(2\pi)^{n/2}\tilde{I}_{n/2-1}(\kappa)},$$ which is independent of the choice of $\bm{\mu}\in\mathcal{S}^{n-1}$, and $\tilde{I}_{n/2-1}(\cdot)$ denotes the modified Bessel functions of the first kind.
Here, the first line comes from Lemma \ref{lem:kldes}. The second line comes from the definition of the KL-divergence and the density function of the von Mises–Fisher distribution, and the third line comes from the definition of $C_n(\kappa)$. The fourth line comes from Lemma \ref{lem:bessel} that
    $$
        \left|\log\frac{\kappa_1^{n/2-1}\tilde{I}_{n/2-1}(\kappa_2)}{\kappa_2^{n/2-1}\tilde{I}_{n/2-1}(\kappa_1)}\right|
        \leq
        \left|\max(\log \text{e}^{\kappa_1-\kappa_2}
        ,\log \text{e}^{\kappa_2-\kappa_1})\right|
        =
        |\kappa_1-\kappa_2|,
    $$
    and the last line comes from Cauchy inequality that
    \begin{align*}
        (\kappa_1\bm{\mu}_1-\kappa_2\bm{\mu}_2)^{\top}\bm{u}
        \leq
        \|\kappa_1\bm{\mu}_1-\kappa_2\bm{\mu}_2\|_2\|\bm{u}\|_2
        =
        \|\kappa_1\bm{\mu}_1-\kappa_2\bm{\mu}_2\|_2.
    \end{align*}
    
Therefore, the $\epsilon^2$-ball centered at some $\tilde{\bm{\theta}}$ with the L$_2$-norm is a subset of the $\epsilon$-ball centered at $\tilde{\bm{\theta}}$ with the total variation metric,
    i.e.,
    \begin{align*}
        \{\bm{\theta}:\|\bm{\theta}-\tilde{\bm{\theta}}\|_2\leq\epsilon^2\}
        \subset
        \{\bm{\theta}:D_{TV}(\bm{\theta},\tilde{\bm{\theta}})\leq\epsilon\}.
    \end{align*} 
    Thus, the $\epsilon$-packing number of $\Theta$ with the total variance metric is bounded by the $\epsilon^2$-packing number of $\Theta$ with the L$_2$-norm, that is,
    \begin{align*}
        \mathcal{N}(\epsilon,\Theta,D_{TV})
        \leq
        \mathcal{N}(\epsilon^2,\Theta,\|\cdot\|_2).
    \end{align*}
    Finally, by Lemma \ref{lem:pacnum}, for all $\epsilon<1$, we have
    \begin{align*}
        \mathcal{N}(\epsilon,\Theta,D_{TV})
        \leq
        \left(1+\frac{\max(2,2\bar{\kappa})}{\epsilon^2}\right)^{n}.
    \end{align*}
\end{proof}

We will use the above two lemmas to bound the numerator of the posterior distribution. First, we derive a lower bound of the denominator of the posterior distribution. And then we will combine those two parts to establish the statement of the posterior distribution.

\begin{lemma}
    \label{lem:post_dom}
    For any $\epsilon>0$, the following inequality holds with probability no less than $1-\frac{4\bar{\kappa}^2}{T\epsilon^2}$,
    \begin{align}
        \int_{\Theta}\frac{\mathbb{P}(\mathcal{D}_T|\bm{\theta})}{\mathbb{P}(\mathcal{D}_T|\bm{\theta}^*)}\mathbb{P}_0(\mathrm{d}\bm{\theta})
        \geq
        \left(\frac{2\epsilon^2}{\bar{\kappa}}\right)^n\cdot\exp(-T\epsilon^2)
    \end{align}
\end{lemma}
\begin{proof}
Let 
    $$
        \Theta_{KL}(\epsilon)
        \coloneqq
        \left\{\bm{\theta}\in\Theta: D_{KL}(\mathbb{P}((\bm{x}^*,\bm{a},b)|\bm{\theta}^*),\mathbb{P}((\bm{x}^*,\bm{a},b)|\bm{\theta}))\leq\epsilon^2\right\}.
    $$
    By inequality \eqref{ieq:klnorm} in Lemma \ref{lem:klnorm}, we have that $\Theta_{KL}(\epsilon)$ contains the $2\epsilon^2$-ball centered at $\bm{\theta}^*$ with the $2$-norm and, therefore, 
    \begin{align}
    \label{ieq:volkl}
        \mathbb{P}_0(\Theta_{KL}(\epsilon))
        \geq
        \left(
            \frac{2\epsilon^2}{\bar{\kappa}}
        \right)^{n}.
    \end{align}
    Next, we show that with probability no less than $1-\exp(-T\epsilon^2/4\bar{\kappa}^2)$, 
    \begin{align}
        \label{ieq:kllb}
            \int_{\Theta_{KL}(\epsilon)}\frac{\mathbb{P}(\mathcal{D}_T|\bm{\theta})}{\mathbb{P}(\mathcal{D}_T|\bm{\theta}^*)}\mathbb{P}_0(\mathrm{d}\bm{\theta})
            \geq
            \exp\left(-T\epsilon^2
            \right)\cdot \mathbb{P}_0(\Theta_{KL}(\epsilon)).
    \end{align}
    Since the range of the density function of $\mathcal{P}_u(\cdot|\bm{\theta})$ is between $[\text{e}^{-\bar{\kappa}},\text{e}^{\bar{\kappa}}]$ for all $\bm{\theta}\in\Theta$, we have that 
    $$
        -2\bar{\kappa}\leq \log\frac{\mathbb{P}((\bm{x}^*,\bm{a},b)|\bm{\theta})}{\mathbb{P}((\bm{x}^*,\bm{a},b)|\bm{\theta}^*)}
        \leq 2\bar{\kappa}
    $$
for all $(\bm{x}^*,\bm{a},b)$. Then, by Hoeffding's inequality (Lemma \ref{HF1}), we have, with probability no more than $\exp(-T\epsilon^2/4\bar{\kappa}^2)$,
    \begin{align}
        \sum\limits_{t=1}^{T}\int_{\Theta_{KL}(\epsilon)}\log\frac{\mathbb{P}((\bm{x}_t^*,\bm{a}_t,b_t)|\bm{\theta})}{\mathbb{P}((\bm{x}_t^*,\bm{a}_t,b_t)|\bm{\theta}^*)}\mathbb{P}_0(\mathrm{d}\bm{\theta}) \nonumber
        \leq&
        T\mathbb{E}_{\bm{\theta}^*}\left(\int_{\Theta_{KL}(\epsilon)}\log\frac{\mathbb{P}((\bm{x}_t^*,\bm{a}_t,b_t)|\bm{\theta})}{\mathbb{P}((\bm{x}_t^*,\bm{a}_t,b_t)|\bm{\theta}^*)}\mathbb{P}_0(\mathrm{d}\bm{\theta})\right)
        -
        2T\epsilon^2\\
        =&
        T\int_{\Theta_{KL}(\epsilon)}D_{KL}(\mathbb{P}((\bm{x}^*,\bm{a},b)|\bm{\theta}^*),\mathbb{P}((\bm{x}^*,\bm{a},b)|\bm{\theta}))\mathbb{P}_0(\mathrm{d}\bm{\theta})-2T\epsilon^2 \nonumber\\
        \leq&
        T\epsilon^2-
        2T\epsilon^2
        =
        -T\epsilon^2 \label{ieq:klhf}
    \end{align}
    where the first line comes directly from Hoeffding's inequality, the second line comes from Fubini's theorem, and the last line comes from the definition of $\Theta_{KL}(\epsilon)$.
    Then, by Jensen's inequality,
    \begin{align}
    \label{ieq:jen}
        \log\int_{\Theta_{KL}(\epsilon)}\frac{\mathbb{P}(\mathcal{D}_T|\bm{\theta})}{\mathbb{P}(\mathcal{D}_T|\bm{\theta}^*)}\frac{\mathbb{P}_0(\mathrm{d}\bm{\theta})}{\mathbb{P}_0(\Theta_{KL}(\epsilon))}
        &=
        \log\int_{\Theta_{KL}(\epsilon)}\prod_{t=1}^{T}\frac{\mathbb{P}((\bm{x}_t^*,\bm{a}_t,b_t)|\bm{\theta})}{\mathbb{P}((\bm{x}_t^*,\bm{a}_t,b_t)|\bm{\theta}^*)}\frac{\mathbb{P}_0(\mathrm{d}\bm{\theta})}{\mathbb{P}_0(\Theta_{KL}(\epsilon))}\\
        &\geq
        \sum\limits_{t=1}^{T}\int_{\Theta_{KL}(\epsilon)}\log\frac{\mathbb{P}((\bm{x}_t^*,\bm{a}_t,b_t)|\bm{\theta})}{\mathbb{P}((\bm{x}_t^*,\bm{a}_t,b_t)|\bm{\theta}^*)}\frac{\mathbb{P}_0(\mathrm{d}\bm{\theta})}{\mathbb{P}_0(\Theta_{KL}(\epsilon))},\nonumber
    \end{align}
Note that $\exp(-T\epsilon^2/4\bar{\kappa}^2)\leq\frac{4\bar{\kappa}^2}{T\epsilon^2}$, we can prove \eqref{ieq:kllb} by combining \eqref{ieq:klhf} with \eqref{ieq:jen}.
    
Finally, with probability no less than $1-\frac{4\bar{\kappa}}{T\epsilon^2}$,
    \begin{align*}
        \int_{\Theta}\frac{\mathbb{P}(\mathcal{D}_T|\bm{\theta})}{\mathbb{P}(\mathcal{D}_T|\bm{\theta}^*)}\mathbb{P}_0(\mathrm{d}\bm{\theta})
        &\geq
        \int_{\Theta_{KL}(\epsilon)}\frac{\mathbb{P}(\mathcal{D}_T|\bm{\theta})}{\mathbb{P}(\mathcal{D}_T|\bm{\theta}^*)}\mathbb{P}_0(\mathrm{d}\bm{\theta})\\
        &\geq
        \exp\left(-T\epsilon^2
        \right)\cdot \mathbb{P}_0(\Theta_{KL}(\epsilon))\\
        &\geq
        \left(\frac{2\epsilon^2}{\bar{\kappa}}\right)^n\cdot\exp(-T\epsilon^2),
    \end{align*}
    where the first inequality is obtained by the non-negativity of the integrand, the second inequality comes from \eqref{ieq:kllb}, and the last line comes from \eqref{ieq:volkl}.
\end{proof}

\subsection{Proof of Theorem \ref{thm:post_conv}}
In this part, we combine three lemmas in the previous section and show Theorem \ref{thm:post_conv}.
\begin{proof}
    Let $$
        \epsilon_T=\max(4,4\bar{\kappa})\frac{\sqrt{n}\cdot\log T}{T^{1/2-\alpha}}.
    $$
    By Lemma \ref{lem:post_dom}, we have, with probability no less than $1-\frac{4\kappa^2}{T\epsilon^2_T}$, the  inequality 
    \begin{align}
        \label{ieq:dom}
        \int_{\Theta}\frac{\mathbb{P}(\mathcal{D}_T|\bm{\theta})}{\mathbb{P}(\mathcal{D}_T|\bm{\theta}^*)}\mathbb{P}_0(\mathrm{d}\bm{\theta})
        \geq
        \left(\frac{2\epsilon_T^2}{\bar{\kappa}}\right)^n\cdot\exp(-T\epsilon_T^2)
        \geq
        \exp(-2T\epsilon^2_T)
    \end{align}
    holds for all $T$ satisfying $T\geq n\log T$. In \eqref{ieq:post_p1}, we will use this inequality to establish a lower bound for the denominator of the posterior distribution.
    
    By Lemma \ref{lem:klnorm}, we have for all $T\geq3$
    \begin{align*}
        \mathcal{N}\left(\frac{\epsilon_T}{2},\Theta,D_{TV}\right)
        \leq
        \left(\frac{5T^{1-2\alpha}}{16 n\cdot \log T}\right)^n
        \leq
        \exp(T\epsilon_T^2),
    \end{align*}
    which gives an upper bound of the packing number and thus verifies the condition of Lemma \ref{lem:testfun}. Then, by Lemma \ref{lem:testfun}, for all $T\geq4$, there exists a function $\phi_T$ mapping the data set $\mathcal{D}_T$ to $[0,1]$, which satisfies
    \begin{align}
        \mathbb{E}_{\bm{\theta}^*}[\phi_T]
        &\leq
        \frac{\exp(T\epsilon_T^2)\exp(-2T\epsilon_T^2)}{1-\exp(-2T\epsilon_T^2)}\leq2\exp(-T\epsilon_T^2)\label{ieq:test1}
        ,\\
        \sup\limits_{D_{TV}(\bm{\theta},\bm{\theta}^*)>2\epsilon_T}\mathbb{E}_{\bm{\theta}}[1-\phi_T]
        &\leq
        \exp(-4T\epsilon_T^{2}).\label{ieq:test2}
    \end{align}
    Let $$
        \tilde{\Theta}_T
        \coloneqq \left\{\bm{\theta}\in \Theta: D_{TV}\left(\prob\left((\bm{x}^*,\bm{a},b)|\bm{\theta}\right), \prob\left((\bm{x}^*,\bm{a},b)|\bm{\theta}^*\right)\right)\le \max\left(8,8{\bar{\kappa}}\right)\frac{\sqrt{n} \cdot \log T}{T^{1/2}}\right\}
    $$
    We then have
    \begin{align}
    \label{ieq:num}
        \mathbb{E}_{\bm{\theta}^*}\left[(1-\phi_T)\int_{\tilde{\Theta}_T^c}\prod\limits_{t=1}^{T}\frac{\mathbb{P}(\mathcal{D}_T|\bm{\theta})}{\mathbb{P}(\mathcal{D}_T|\bm{\theta}^*)}\mathbb{P}_0(\mathrm{d}\bm{\theta})\right]\nonumber
        &=
        \int_{\tilde{\Theta}_T^c}\mathbb{E}_{\bm{\theta}^*}\left[(1-\phi_T)\prod\limits_{t=1}^{T}\frac{\mathbb{P}(\mathcal{D}_T|\bm{\theta})}{\mathbb{P}(\mathcal{D}_T|\bm{\theta}^*)}\right]\mathbb{P}_0(\mathrm{d}\bm{\theta})\\
        &=
        \int_{\tilde{\Theta}_T^c}\mathbb{E}_{\bm{\theta}}(1-\phi_T)\mathbb{P}_0(\mathrm{d}\bm{\theta})\\
        &\leq
        \exp(-4T\epsilon_T^2),\nonumber
    \end{align}
    where the first line is obtained by Fubini's theorem, the second line is obtained directly by computing the inner integral, and the last line comes from the definition of $\tilde{\Theta}_T$ and inequality \eqref{ieq:test2}. Denote the low probability event corresponding to inequality \eqref{ieq:dom} as $\mathcal{E}_T$. By combining \eqref{ieq:dom} and \eqref{ieq:num}, we have
    \begin{align}
    \label{ieq:post_p1}
        \mathbb{E}_{\bm{\theta}^*}
        \left[
            \mathbb{P}_T(\tilde{\Theta}_T^c)(1-\phi_T)I_{\mathcal{E}_T}
        \right]
        &=
        \mathbb{E}_{\bm{\theta}^*}\left[
            \frac{(1-\phi_T)I_{A_T}\int_{\tilde{\Theta}_T^c}\frac{\mathbb{P}(\mathcal{D}_T|\bm{\theta})}{\mathbb{P}(\mathcal{D}_T|\bm{\theta}^*)}\mathbb{P}(\mathrm{d}\bm{\theta})}{\int_{\Theta}\frac{\mathbb{P}(\mathcal{D}_T|\bm{\theta})}{\mathbb{P}(\mathcal{D}_T|\bm{\theta}^*)}\mathbb{P}(\mathrm{d}\bm{\theta})}
        \right]\nonumber
        \\
        &\leq
        \exp(-4T\epsilon_T^2)\exp(2T\epsilon_T^2)=
        \exp(-2T\epsilon_T^2),
    \end{align}
    where the first equality comes from the definition of the posterior distribution, and the second line is obtained by plugging in \eqref{ieq:dom} and \eqref{ieq:num}. 
    
    Finally, we have
    \begin{align}
    \label{post_TV_conv}
        \mathbb{E}_{\bm{\theta}^*}\left[
            \mathbb{P}_T(\tilde{\Theta}_T)
        \right]
        &\geq
        1-\mathbb{E}_{\bm{\theta}^*}\left[
            (1-\phi_T)I_{\mathcal{E}_T}\mathbb{P}_T(\tilde{\Theta}_T^c]
        \right]
        -
        \mathbb{E}_{\bm{\theta}^*}\left[
            \phi_T\mathbb{P}_T(\tilde{\Theta}_T^c)
        \right]
        -
        \mathbb{E}_{\bm{\theta}^*}\left[
            I_{\mathcal{E}_T}\mathbb{P}_T(\tilde{\Theta}_T^c)
        \right]\nonumber\\
        &\geq
        1-\mathbb{E}_{\bm{\theta}^*}\left[
            (1-\phi_T)I_{\mathcal{E}_T}\mathbb{P}_T(\tilde{\Theta}_T^c)
        \right]
        -
        \mathbb{E}_{\bm{\theta}^*}\left[
            \phi_T
        \right]
        -
        \mathbb{E}_{\bm{\theta}^*}\left[
            I_{\mathcal{E}_T}
        \right]\\
        &\geq
        1-2\exp(-T\epsilon_T^2)-\frac{4\bar{\kappa}^2}{T\epsilon_T^2}\nonumber\\
        &\geq
        1-\frac{2}{T}-\frac{1}{4T^{2\alpha}\log^2 T}\nonumber,
    \end{align}
    where the first line comes from the fact that the posterior probability is bounded by 1, the second line comes from \eqref{ieq:test1} and \eqref{ieq:post_p1}, and the last line comes from the definition of $\epsilon_T$. The inequality above also indicates that $\mathbb{P}_T(\tilde{\Theta}_T)$ converges to 1 in $L_1$ norm, which implies the convergence in probability. In fact, with a little modification for the bound in Lemma \ref{lem:post_dom}, we can get rid of the parameter $\alpha$ in the definition of $\epsilon_T$, and the expectation can be bounded by $1-\frac{3}{T}$ from below.
    
    To obtained a similar result for the $L_2$ Wasserstein distance, we only need to establish the relationship between $\Theta_T$ and $\tilde{\Theta}_T$. Notice that $\mathcal{P}_{\bm{a},b}$ is independent of $\mathcal{P}_u(\cdot|\bm{\theta})$ and the maximum distance between two points in $[0,1]^{n}$ is no larger than $\sqrt{n}$. We have
    $$
        W_2(\mathcal{P}((\bm{x}^*,\bm{a},b)|\bm{\theta}),\mathcal{P}((\bm{x}^*,\bm{a},b)|\bm{\theta}^*))
        \leq
        \sqrt{n}D_{TV}(\bm{\theta},\bm{\theta}^*).
    $$
Thus, the result of \eqref{post_TV_conv} also holds for the $L_2$ Wasserstein distance metric if we have an additional $\sqrt{n}$ factor in the definition of $\epsilon_T$ and $\tilde{\Theta}_T$, which is exactly $\Theta_T$ in Theorem \ref{thm:post_conv}.
\end{proof}

\subsection{Proof of Corollary \ref{coro:Stoch}}
\begin{proof}
The convergence of the posterior distribution to the point mass distribution can be directly obtained by Doob's consistency theorem (Lemma \ref{lem:doob}). Here, we only show the second part, i.e., the upper bound of the posterior expectation $\mathbb{E}_T[\|\bm{\theta}_T-\bm{\theta}^*\|_2]$. Without loss of generality, we assume that $L\leq1$.

From the proof of Theorem \ref{thm:post_conv}, we have that, there exists a high probability event $\tilde{\mathcal{E}}_T$ satisfying $$\mathbb{P}_{\bm{\theta}^*}(\tilde{\mathcal{E}}_T)\geq1-\frac{1}{4 n\log^2 T \cdot T^{2\alpha}}$$ such that
\begin{align*}
    \mathbb{E}_{\bm{\theta}^*}
    \left[
        \mathbb{P}_T(\Theta_T^c)
        I_{\tilde{\mathcal{E}}_T}
    \right]
    \leq
    \frac{2}{T}.
\end{align*}

Then, by Markov's inequality, we have for all $T\geq10$
\begin{align}
    \label{ieq:mudiff_p1}
    \mathbb{P}_{\bm{\theta}^*}\left(\mathbb{P}_T(\Theta_T^c)>\frac{n\log T}{2L T^{1/2-\alpha}}\right)\nonumber
    &\leq
    \mathbb{P}_{\bm{\theta}^*}(\tilde{\mathcal{E}}_T^c)+\mathbb{P}_{\bm{\theta}^*}\left(\tilde{\mathcal{E}}_t\cap\left(\mathbb{P}_T(\Theta_T^c)>\frac{n\log T}{2L T^{1/2-\alpha}}\right)\right)\\
    &\leq
    \frac{1}{4n \log^2 T \cdot T^{2\alpha}}+
    \frac{4L}{n\log T \cdot T^{1/2+\alpha}}\\
    &\leq
    \frac{1}{\log^2 T \cdot T^{2\alpha}},\nonumber
\end{align}
where the first line comes from a decomposition of the event, the second line comes from Markov's inequality, and the last line comes from the fact that $n\geq1$ and $L\leq1$.

Next, from the condition \eqref{eqn:equiv},
$$
    \mathcal{W}\left(\prob\left((\bm{x}^*_t,\bm{a}_t,b)|\bm{\theta}\right), \prob\left((\bm{x}^*_t,\bm{a}_t,b)|\bm{\theta}^*\right)\right)
\geq L\cdot \|\bm{\theta}- \bm{\theta}^*\|_2.
$$
Consequently, for any $\bm{\theta}\in\Theta_T$, we have
\begin{align}
    \label{ieq:w2norm}
    \|\bm{\theta}-\bm{\theta}^*\|_2
    \leq
    \max\left(8,8\bar{\kappa}\right)\frac{n \cdot \log T}{L\cdot T^{1/2-\alpha}}.
\end{align}
Combining \eqref{ieq:mudiff_p1} and \eqref{ieq:w2norm}, we have with probability no less than $\frac{1}{\log^2 T\cdot T^{2\alpha}}$
\begin{align*}
    \mathbb{E}_T\left[\|\bm{\theta}_T-\bm{\theta}^*\|_2\right]
    &\leq
    \max\left(8,8\bar{\kappa}\right)\frac{n \cdot \log T}{L\cdot T^{1/2-\alpha}}\cdot \mathbb{P}_T(\Theta_T)
    +
    2\bar{\kappa}\cdot\mathbb{P}_T(\Theta_T^c)\\
    &\leq
    \max\left(8,8\bar{\kappa}\right)\frac{n \cdot \log T}{L\cdot T^{1/2-\alpha}}
    +
    \max\left(1,1\bar{\kappa}\right)\frac{n \cdot \log T}{L\cdot T^{1/2-\alpha}}=
    \max\left(9,9\bar{\kappa}\right)\frac{n \cdot \log T}{L\cdot T^{1/2-\alpha}}.
\end{align*}
Here, the first inequality is obtained by the fact the the maximum distance between two different parameters are bounded by $2\bar{\kappa}$, and the second inequality is obtained by \eqref{ieq:w2norm}.
\end{proof}

\subsection{Proof of Corollary \ref{predict_coro_gaussian}}

\begin{proof}
    Recall the definition of $\tilde{\Theta}_T$
    $$
        \tilde{\Theta}_T
        = \left\{\bm{\theta}\in \Theta: D_{TV}\left(\prob\left((\bm{x}^*,\bm{a},b)|\bm{\theta}\right), \prob\left((\bm{x}^*,\bm{a},b)|\bm{\theta}^*\right)\right)\le \max\left(8,8{\bar{\kappa}}\right)\frac{\sqrt{n} \cdot \log T}{T^{1/2-\alpha}}\right\}.
    $$
In a similar way as \eqref{ieq:mudiff_p1}, we obtain
    \begin{align*}
    \mathbb{P}_{\bm{\theta}^*}\left(\mathbb{P}_T(\tilde{\Theta}_T^c)>\frac{\sqrt{n}\log T}{2 T^{1/2-\alpha}}\right)
    &\leq
    \frac{1}{\log^2 T \cdot T^{2\alpha}}.
    \end{align*}
    Then we utilize the total variation distance to bound the difference between optimal solutions. From the integral representation of the total variation distance, we have, for any $\bm{\theta}$,
    \begin{align}
    \label{eq:tveq}
        D_{TV}\left(\prob\left((\bm{x}^*,\bm{a},b)|\bm{\theta}\right), \prob\left((\bm{x}^*,\bm{a},b)|\bm{\theta}^*\right)\right)
        =
        \frac{1}{2}\int_{(\bm{x}^*,\bm{a},{b})} \left|1-\frac{\prob\left((\bm{x}^*,\bm{a},b)|\bm{\theta}^*\right)}{\prob\left((\bm{x}^*,\bm{a},b)|\bm{\theta}\right)}\right| \prob\left((\mathrm{d}\bm{x}^*,\mathrm{d}\bm{a},\mathrm{d}{b})|\bm{\theta}\right).
    \end{align}
    On the right hand side, the ratio $\frac{\prob\left((\bm{x}^*,\bm{a},b)|\bm{\theta}\right)}{\prob\left((\bm{x}^*,\bm{a},b)|\bm{\theta}^*\right)}$ is the probability that the optimal solutions (corresponding to $\bm{\theta}^*$ and $\bm{\theta}$) coincide for a fixed pair of $(\bm{a},b)$. Thus, the integration calculates the probability that the the optimal solution corresponding to $\bm{\theta}$ is different from the optimal solution corresponding to $\bm{\theta}^*$. 
    Let $\bm{\theta}_T$ be a random parameter drawn from the posterior distribution. 
    Denote $\tilde{\bm{x}}^*$ as the optimal solution corresponding to $\bm{\theta}_T$ given $(\bm{a},b)$, and ${\bm{x}}^*$ as the optimal solution corresponding to $\bm{\theta}^*$ given $(\bm{a},b)$. We have
    \begin{align}
    \label{ieq:bddx}
        \mathbb{P}_T({\bm{x}}^*\not=\tilde{\bm{x}}^*)
        \leq
        2D_{TV}\left(\prob\left((\bm{x}^*,\bm{a},b)|\bm{\theta}_T\right), \prob\left((\bm{x}^*,\bm{a},b)|\bm{\theta}^*\right)\right).
    \end{align}
    Thus, with probability no less than $1- \frac{1}{\log^2 T \cdot T^{2\alpha}}$, we have
    \begin{align*}
        \mathbb{E}[\|\bm{x}^*-\tilde{\bm{x}}^*\|_2]
        &\leq
        \sqrt{n}\mathbb{P}_T({\bm{x}}^*\not=\tilde{\bm{x}}^*)\\
        &\leq
        2\sqrt{n}D_{TV}\left(\prob\left((\bm{x}^*,\bm{a},b)|\bm{\theta}\right), \prob\left((\bm{x}^*,\bm{a},b)|\bm{\theta}^*\right)\right),\\
        &\leq
        \max\left(16,16{\bar{\kappa}}\right)\frac{\sqrt{n} \cdot \log T}{T^{1/2-\alpha}}
    \end{align*}
    where the first line comes from the fact that the maximum distance between any two solutions is bounded by $\sqrt{n}$, the second line comes from \eqref{ieq:bddx}, and the last line comes from the definition of $\tilde{\Theta}_T$.
    
We remark that the statement is also true for general distributions by replacing the probability ratio in \eqref{eq:tveq} with the Radon–Nikodym derivative.
    
\end{proof}

\section{Proof of Section \ref{sec_delta}}
In this section, we prove the results in Section \ref{sec_delta}.

\subsection{Proof of Lemma \ref{lemma_linear}}

\begin{proof}
For an observation of $(\bm{x}^*,\bm{a}, b)$, let
$$
\mathcal{U}\coloneqq\left\{\bm{u}\in \mathcal{S}^{n-1}: \bm{x}^* \text{ is an optimal solution of } \text{LP}(\bm{u}, \bm{a}, b)\right\}.
$$
From the LP's optimality conditions, we know that $\bm{u}\in\mathcal{U}$ if and only if the following three inequalities hold:
\begin{align}
    -\min_{\{i:x_i^*>0\}}\left(\frac{u_i}{a_i}\right)&\leq 0, \label{opc1}\\
    \max_{\{i:x_i^*=0,u_i\not=0\}}\frac{u_i}{a_i}&\leq0, \text{if $\bm{a}^\top\bm{x}^*<b$} \label{opc2}\\
    \max_{\{i:x_i^*=0,u_i>0\}}\frac{u_i}{a_i}-\min_{\{i:x_i>0\}}\frac{u_i}{a_i} &\leq 0, \text{ if $\bm{a}^\top\bm{x}^*=b$} \label{opc3}.
\end{align}
All the three inequalities can be expressed by linear constraints, so we finish the proof.
\end{proof}

\subsection{Proof of Proposition \ref{prop:sim}}
\begin{proof}
The proof is a direct application of Hoeffding's inequality. Denote $X_t$ as the indicator function of $\bm{u}_t\not=\bm{u}^*$. By Hoeffding's inequality, we have
\begin{align*}
    \mathbb{P}\left(\frac{1}{T}\sum\limits_{t=1}^{T}X_t\leq\delta+\frac{\log T}{\sqrt{T}}\right)
    \leq
    \frac{1}{T}.
\end{align*}
Then, it is sufficient to show that 
\begin{align*}
    \max_{\bm{u}\in\mathcal{S}^{n-1}} \left\vert\frac{1}{T}\sum\limits_{t=1}^{T} I_{\mathcal{U}_t}(\bm{u})-\frac{1}{T}\sum\limits_{t=1}^{T}I_{\bar{\mathcal{U}}_t}(\bm{u})\right\vert
    \leq
    \frac{1}{T}\sum\limits_{t=1}^{T}X_t.
\end{align*}\
To see this, we have
\begin{align*}
    \max_{\bm{u}\in\mathcal{S}^{n-1}} \left\vert\frac{1}{T}\sum\limits_{t=1}^{T} I_{\mathcal{U}_t}(\bm{u})-\frac{1}{T}\sum\limits_{t=1}^{T}I_{\bar{\mathcal{U}}_t}(\bm{u})\right\vert
    &\leq
    \max_{\bm{u}\in\mathcal{S}^{n-1}} \left\vert\frac{1}{T}\sum\limits_{\{t:X_t=1\}} I_{\mathcal{U}_t}(\bm{u})-\frac{1}{T}\sum\limits_{\{t:X_t=1\}}I_{\bar{\mathcal{U}}_t}(\bm{u})\right\vert\\
    &\leq
    \max_{\bm{u}\in\mathcal{S}^{n-1}} \frac{1}{T}\sum\limits_{\{t:X_t=1\}}\left\vert I_{\mathcal{U}_t}(\bm{u})-I_{\bar{\mathcal{U}}_t}(\bm{u})\right\vert\\
    &\leq
    \max_{\bm{u}\in\mathcal{S}^{n-1}} \frac{1}{T}\sum\limits_{\{t:X_t=1\}}1\\
    &=
    \frac{1}{T}\sum\limits_{t=1}^{T}X_t,
\end{align*}
where the first inequality is obtained by the fact that $I_{\mathcal{U}_t}(\bm{u})=I_{\bar{\mathcal{U}}_t}(\bm{u})$ if $X_t=0$, the second line comes from Jensen's inequality for the absolute value function, and the last two lines come directly from the property of indicator functions.
\end{proof}

\subsection{Proof of Proposition \ref{prop_generalize}}
\begin{proof}
    In this part, we will show a stronger statement that, for all $\gamma>0$,
    $$
          \max_{\bm{u}:\|\bm{u}\|_2\leq1} \  -\mathrm{Acc}(\bm{u}) + 
            \frac{1}{T}\sum\limits_{t=1}^{T}I_{\mathcal{U}_{t}(\gamma)}(\bm{u})
            \le 
            4\sqrt{\frac{\log(T)}{{\underline{a}^2\gamma^2 T}}}+
            6\sqrt{\frac{\log(T/\epsilon)}{{T}}},
    $$
    where the uniform bound holds for all $\bm{u}$ in the unit ball (instead of unit sphere). 

    
We will utilize Lemma \ref{lem:pacbay} to prove the above inequality. For any $\bm{u}_0$ in the unit ball, let $\mathbb{Q}=\mathcal{N}(0,\tau^2\bm{I}_n)$ and $\mathbb{Q}_0=\mathcal{N}(\bm{u}_0,\tau^2\bm{I}_n)$ be two normal distributions over the estimated parameter space, where $\bm{I}_n$ is the $n$-dimensional identity matrix and $\tau$ is a constant to be determined. We remark that our choice of $\mathbb{Q}_0$ and $\mathbb{Q}$ will not affect the distribution of $(\bm{u}_t,\bm{a}_t,b)$, which depends on $\mathcal{P}_{\bm{a},b}$, $\mathcal{P}_{u}'$, $\delta$, and $\bm{u}^*$. Thus, by Lemma \ref{lem:pacbay}, the following inequality holds with probability no less than $1-\epsilon$,
    \begin{align}
    \label{ieq:pacb}
        \mathbb{E}_{\mathbb{Q}_0}\left[\mathbb{E}\left[I_{(\mathcal{U}(\gamma))^c}({\bm{u}})\right]\right]
        \leq
        \frac{1}{T}\mathbb{E}_{\mathbb{Q}_0}\left[\sum\limits_{t=1}^{T}I_{(\mathcal{U}_t(\gamma))^c}({\bm{u}})\right]
        +
        \sqrt{\frac{D_{KL}(\mathbb{Q},\mathbb{Q}_0)+\log\frac{T}{\epsilon}+2}{2T-1}},
    \end{align}
where $\mathcal{E}^c$ denotes the complement of a set $\mathcal{E}.$ We note that on the left-hand-side, the inner expectation is taken with respect to the indicator function, and the outer expectation is taken with respect to $\bm{u}\sim\mathbb{Q}_0.$

Then we have
    \begin{align}
        \label{ieq:pacbay}
\mathbb{E}_{\mathbb{Q}_0}\left[\mathbb{E}\left[I_{(\mathcal{U}(\gamma))^c}({\bm{u}})\right]\right]
        &\leq
         \frac{1}{T}\mathbb{E}_{\mathbb{Q}_0}\left[\sum\limits_{t=1}^{T}I_{(\mathcal{U}_t(\gamma))^c}({\bm{u}})\right]
        +
        \sqrt{\frac{\frac{\|{\bm{u}}_0\|_2^2}{2\tau^2}+\log\frac{T}{\epsilon}+2}{2T-1}}\\
        &\leq
        \frac{1}{T}\mathbb{E}_{\mathbb{Q}_0}\left[\sum\limits_{t=1}^{T}I_{(\mathcal{U}_t(\gamma))^c}({\bm{u}})\right]
        +
        2\sqrt{\frac{\frac{\|{\bm{u}}_0\|_2^2}{2\tau^2}+\log\frac{T}{\epsilon}}{T}}.\nonumber
    \end{align}
Here, the first line is obtained by  calculating of the KL-divergence between two Gaussian distributions, and the second line is a further simplification of the second line.
    
Now we analyze the left-hand-side and show that
    \begin{align}
    \label{ieq:0gamma}
        I_{(\mathcal{U}(0))^c}({\bm{u}}_0)-2\exp\left(-\frac{\underline{a}^2\gamma^2}{2\tau^2}\right)
        \leq
        \mathbb{E}_{\mathbb{Q}_0}\left[I_{(\mathcal{U}(\gamma))^c}({\bm{u}})\right]
        \leq
        I_{(\mathcal{U}(2\gamma))^c}({\bm{u}}_0)+2\exp\left(-\frac{\underline{a}^2\gamma^2}{2\tau^2}\right).
    \end{align}

To see this, for a fixed $(\bm{x}^*,\bm{a},b)$, if $I_{\mathcal{U}(0)^c}({\bm{u}}_0)=1$, at least one inequalities from \eqref{opc1} to \eqref{opc3} is violated. If \eqref{opc1} does not hold for ${\bm{u}}_0$ while $I_{(\mathcal{U}(\gamma))^c}(\bm{u})=0$, there exists at least one $i\in\{1,...,n\}$ such that 
    $$
    \frac{({\bm{u}}_0)_i}{a_i}\geq0,\ 
    \frac{u_i}{a_i}<-\gamma.
    $$ The corresponding probability is no less than $\mathbb{P}((\hat{\bm{u}}-\hat{\bm{u}}_0)_i\leq -\underline{a}\gamma)$, which is bounded by $\exp(-\frac{\underline{a}^2\gamma^2}{2\tau^2})$. Following the same analysis for \eqref{opc2} and \eqref{opc3}, we have with probability no more than $2\exp(-\frac{\underline{a}^2\gamma^2}{2\tau^2})$, $$I_{\mathcal{U}^c}({\bm{u}})=0, \text{ while } I_{(\mathcal{U}(\gamma))^c}({\bm{u}}_0)=1,$$
    which gives the left part of \eqref{ieq:0gamma}. The right part follows the same analysis. 

    Thus, let $\tau^2=\frac{\underline{a}^2\gamma^2}{2\log T}$. From \eqref{ieq:0gamma} and \eqref{ieq:pacbay}, we have, with probability no less than $1-\epsilon$, the following inequalities hold simultaneously for all ${\bm{u}}_0$ in the unit ball,
    \begin{align}
    \label{ieq:gen2g}
        \mathbb{E}[I_{\mathcal{U}^c}({\bm{u}}_0)]
        &\leq
        \frac{1}{T}\sum\limits_{t=1}^{T}I_{(\mathcal{U}(2\gamma))^c}({\bm{u}}_0)
        +
        4\exp(-\frac{\underline{a}^2\gamma^2}{2\tau^2})
        +
        2\sqrt{\frac{\frac{\|{\bm{u}}_0\|_2^2}{2\tau^2}+\log\frac{T}{\epsilon}}{T}\nonumber}\\
        &\leq
        \frac{1}{T}\sum\limits_{t=1}^{T}I_{(\mathcal{U}(2\gamma))^c}({\bm{u}}_0)
        +
        \frac{4}{T}
        +
        2\sqrt{\frac{\log T}{\underline{a}^2\gamma^2T}+\frac{\log(T/\epsilon)}{T}}\\
        &\leq
        \frac{1}{T}\sum\limits_{t=1}^{T}I_{(\mathcal{U}(2\gamma))^c}({\bm{u}}_0)
        +
        2\sqrt{\frac{\log T}{\underline{a}^2\gamma^2T}}+6\sqrt{\frac{\log(T/\epsilon)}{T}},\nonumber
    \end{align}
    where the first inequality is obtained by plugging \eqref{ieq:0gamma} into both sides of \eqref{ieq:pacbay}, the second line is obtained by plugging the value of $\tau$ into the inequality, and the last line is obtained by the convexity of the square root function. Finally, with probability no less than $1-\epsilon$
    \begin{align*}
        \sup_{\{\bm{u}:\|\bm{u}\|_2\leq1\}}
        -\text{Acc}({\bm{u}})+\frac{1}{T}
        \sum\limits_{t=1}^{T}I_{\mathcal{U}_t(\gamma)}({\bm{u}})
        &=
        \sup_{\{\bm{u}:\|\bm{u}\|_2\leq1\}}
        \mathbb{E}[I_{\mathcal{U}^c}({\bm{u}})]-\frac{1}{T}
        \sum\limits_{t=1}^{T}I_{(\mathcal{U}_t(\gamma))^c}({\bm{u}})\\
        &\leq
        4\sqrt{\frac{\log T}{\underline{a}^2\gamma^2T}}+6\sqrt{\frac{\log(T/\epsilon)}{T}},
    \end{align*}
    where the first line comes from the fact that $I_{\mathcal{A}}(\bm{u})=1-I_{\mathcal{A}}(\bm{u})$ holds for any point $\bm{u}$ and set $\mathcal{A}$, and the second line comes from \eqref{ieq:gen2g} with replacing $2\gamma$ by $\gamma$.
\end{proof}

\subsection{Proof of Theorem \ref{thm_generalize}}

In this part, we assume that the optimal solution $\bm{x}^*_t$ satisfies $$(\bm{x}^*_t)_i=0 \text{ if } (\bm{u}_t)_i\leq0 \text{ for all $t=1,...,T$ and $i=1,...,n$}.$$ This assumption is natural in that if purchasing the $i$-th item brings no positive utility, the customer will not purchase the item.
    
The proof is divided into two parts. In the first part, we show that, if $\gamma$ is sufficiently small, with high probability, the $\gamma$ margin indicator $I_{\mathcal{U}(\gamma)}(\bm{u}^*)$ is almost same as $I_{\mathcal{U}}(\bm{u}^*)$. In the second part, we combine the first part and Proposition \ref{prop_generalize} to draw the conclusion. Intuitively, the smaller $\gamma$ is, the smaller the difference between $I_{\mathcal{U}(\gamma)}(\bm{u}^*)$ and $I_{\mathcal{U}}(\bm{u}^*)$ we will have, and the worse the bound in Proposition \ref{prop_generalize} will be. Then the second part trades off between these two aspects and chooses the best $\gamma$. The formal proof is stated as below.
    
\begin{proof}    
    Let
    $$
        \bar{\mathcal{U}}\coloneqq\left\{\bm{u}\in \mathcal{S}^{n-1}: \bar{\bm{x}}_t^* \text{ is an optimal solution of } \text{LP}(\bm{u}, \bm{a}, b)\right\},
    $$    
    where $\bar{\bm{x}}_t^*$ is an optimal solution of $\text{LP}(\bm{u}^*, \bm{a}, b)$.
    Now, let us show that $I_{\bar{\mathcal{U}}}(\bm{u}^*)=I_{\bar{\mathcal{U}}(\gamma)}(\bm{u}^*)$ holds with probability no less than $1-\frac{4n^2\gamma}{\underline{a}\min\limits_{i:u_i^*\not=0}|u_i^*|}$.
    By definition of $\bar{\mathcal{U}}$ and the assumption that $a_i\leq 1$ for all $i=1,...,n$, we have $\bm{x}^*\in\bar{\mathcal{U}}$ and
    \begin{align*}
        -\min_{\{i:\bar{x}^*_i>0\}}\left(\frac{u_i^*}{a_i}\right)&\leq
        -\min\limits_{i:u_i^*\not=0}|u_i^*|,\\
        \max_{\{i:\bar{x}^*_i=0,u_i^*\not=0\}}\left(\frac{u_i^*}{a_i}\right)&\leq
        -
        \min\limits_{i:u_i^*\not=0}|u_i^*|,
    \end{align*}
    for any $\bm{a}$, $b$ and $\bm{u}^*$. Moreover, for any $i,j=1,...,n$ such that $u_i^*$, we have
    $$
        \left\vert\frac{u_j^*}{a_j}-\frac{u_i^*}{a_i}\right\vert\leq
        \gamma
        \Leftrightarrow
        \frac{u_i^*a_j^*}{u_j^*}-\frac{a_ia_j}{u_j^*}\gamma\leq a_i\leq\frac{u_i^*a_j^*}{u_j^*}+\frac{a_ia_j}{u_j^*}\gamma,
    $$
    which happens with probability no more than $\frac{4\bar{p}\gamma}{\min\limits_{i:u_i^*\not=0}|u_i^*|}$. Since there are at most $n^2$ index pairs, with probability no less than $1-\frac{4n^2\bar{p}\gamma}{\min\limits_{i:u_i^*\not=0}|u_i^*|}$, we have
    \begin{align}
    \label{ieq:goodgamma}
        I_{\bar{\mathcal{U}}(\gamma)}(\bm{u}^*)
        =
        I_{\bar{\mathcal{U}}}(\bm{u}^*),
        \text{ for all $b$, $\gamma\leq\min\limits_{i:u_i^*\not=0}|u_i^*|,$ and fixed $\bm{u}^*\in\mathcal{S}^{n-1}$}.
    \end{align}
Then, similar to the proof of Lemma \ref{lemma_linear} with Hoeffding's inequality, we have 
    \begin{align*}
        \mathbb{P}\left(\left|\sum\limits_{t=1}^{T}I_{\bar{\mathcal{U}}_t(\gamma)}(\bm{u}^*)-\sum\limits_{t=1}^{T}I_{\bar{\mathcal{U}}_t}(\bm{u}^*)\right|>2n^2\bar{p}\gamma+\frac{\log T}{\sqrt{T}}\right)
        &\leq
        2\exp\left(-\frac{T\log T}{T}\right)=\frac{2}{T},
    \end{align*}
    where the randomness in the above inequalities comes only from $(\bm{a},b)$. Here, the inequality comes from the analysis of \eqref{ieq:goodgamma} and Hoeffding's inequality.

    Next, we combine the above analysis and Proposition \ref{prop_generalize}. Let $\hat{\bm{u}}$ be the optimal solution of the problem $OPT_{\delta}(\gamma)$. Then, we have, for any $\gamma\leq\min\limits_{i:u_i^*\not=0}|u_i^*|$, with probability no more than $1-\epsilon-\frac{3}{T}$,
    \begin{align}
    \label{ieq;accbd}
        \text{Acc}(\hat{\bm{u}})&\geq
        \sum\limits_{t=1}^{T}I_{\mathcal{U}_t(\gamma)}(\hat{\bm{u}})-4\sqrt{\frac{\log T}{\underline{T}}}-6\sqrt{\frac{\log (T/\epsilon)}{T}}\nonumber\\
        &\geq
        I_{\mathcal{U}_t(\gamma)}({\bm{u}^*})-4\sqrt{\frac{\log T}{\underline{a}^2\gamma^2T}}-6\sqrt{\frac{\log (T/\epsilon)}{T}}\nonumber\\
        &\geq
        I_{\bar{\mathcal{U}}_t(\gamma)}({\bm{u}^*})-\delta-4\sqrt{\frac{\log T}{\underline{a}^2\gamma^2T}}-7\sqrt{\frac{\log (T/\epsilon)}{T}}\\
        &\geq
        I_{\bar{\mathcal{U}}_t}({\bm{u}^*})-\delta-\frac{4n^2\bar{p}\gamma}{\min\limits_{i:u_i^*\not=0}|u_i^*|}-4\sqrt{\frac{\log T}{\underline{a}^2\gamma^2T}}-8\sqrt{\frac{\log (T/\epsilon)}{T}}\nonumber\\
        &\geq
        1-\delta-\frac{4n^2\bar{p}\gamma}{\min\limits_{i:u_i^*\not=0}|u_i^*|}-4\sqrt{\frac{\log T}{\underline{a}^2\gamma^2T}}-8\sqrt{\frac{\log (T/\epsilon)}{T}},\nonumber
    \end{align}
    where the first inequality is obtained by Proposition \ref{prop_generalize}, the second line is obtained by the optimality of $\hat{\bm{u}}$, the third line is obtained by a similar statement as Proposition \ref{prop:sim}, the fourth line comes from \eqref{ieq:goodgamma}, and the last line is obtained by the fact that $I_{\bar{\mathcal{U}}_t}({\bm{u}^*})=1$ for all $t$. Finally, plug $\epsilon=\frac{1}{T}$ and $\gamma = \frac{1}{4n^2T^{1/4}}$ into \eqref{ieq;accbd}. If $T$ is sufficient large such that $\frac{1}{4n^2T^{1/4}}\leq\min\limits_{i:u_i^*\not=0}|u_i^*|$, we have,  with probability no less than $1-\frac{4}{T}$,
    \begin{align*}
        \text{Acc}(\hat{\bm{u}})
        &\geq
        1-\delta-
        \frac{\bar{p}}{T^{1/4}}-
        4\sqrt{\frac{n^4\log T}{\underline{a}^2T}}
        -
        16\sqrt{\frac{\log T}{T}}
        \geq
        1-\delta-\frac{\bar{p}}{\min\limits_{i:u_i^*\not=0}|u_i^*|\cdot T^{1/4}}-20\frac{n^2\log T}{\underline{a}T^{1/4}}.
    \end{align*}

    We remark that we can further reduce the dependency of $n$ by setting $\gamma=\frac{1}{4nT^{1/4}}$.
\end{proof}

\section{Gaussian setting with Known Concentration Parameter}

In this section, we revisit the Gaussian setting and discuss two methods to learn the mean vector $\bm{\mu}$ when the concentration parameter $\kappa$ is known. Specifically, we consider the following two cases : (i) We can design the constraint pair $(\bm{a}_t,b_t)$ for all $t=1,...,T$, (ii) We can only design $\bm{a}_t$ for all $t$ while we know a lower bound $\underline{b}$ such that $b_t\geq\underline{b}$ for all $t$. The basic idea is to estimate $\mathbb{P}(u_i>0|\bm{\mu})$ for every $i=1,...,n$ and to use the idea of ``moment matching'' to identify $\bm{\mu}$.

\paragraph{Constraint Design for the First Case}\

For the first case, we set $\bm{a}_t=(1,...,1)\in\mathbb{R}^{n}$ and $b_t=n$ for all $t=1,...,T$. Then, at each time $t$, we have $(\bm{x}^*_t)_i=1$ if and only if $(\bm{u}_t)_i>0$. We can then estimate $\mathbb{P}(u_i>0\vert\bm{\mu})$ by the sample average mean
$$
   \hat{p}_i \coloneqq \frac{1}{T} \cdot\#\{t=1,...,T:(\bm{x}^*_t)_i=1\}.
$$
Next, we estimate $\bm{\mu}$ based on $\hat{p}_i$. 

For any $i=1,..,n$, let $\phi$ be an angle in $[0,\pi]$ satisfying $\cos\phi=\mu_i$. By Lemma 1 from \cite{romanazzi2014discriminant}, we have 
\begin{align}
    \label{von_density}
    \mathbb{P}(u_i>0\vert\bm{\mu})
    =
    \int_{0}^{\pi/2}
    \left(\frac{\kappa}{2\pi}\right)^{1/2}\frac{\sin^{(n-1)/2}\psi\cdot \tilde{I}_{(n-3)/2}(\kappa\sin\phi\sin\psi)}{\sin^{(n-3)/2}\phi\cdot\tilde{I}_{n/2-1}(\kappa)}\cdot\exp\left(\kappa\cos\phi\cos\psi\right)
    \mathrm{d}\psi,
\end{align}
which depends only on $\mu_i$. Thus, with slight abuse of notation, we denote $\mathbb{P}(u_i>0|\mu_i)$ as the probability in \eqref{von_density}. Moreover, if $n=2$, we can have that the above function is strictly increasing. Then, by Lemma 1 from \cite{romanazzi2014discriminant} with the induction method, we can show that the above function is strictly increasing with respect to $\mu_i$ for every fixed $n$ and $\kappa$. For any fixed $n$ and $\kappa$, we first numerically compute \eqref{von_density}. Then, a natural estimate of $\mu_i$ is
$$
    \hat{\mu}_i = \argmin_{\mu'_i\in[-1,1]}\vert\mathbb{P}(u_i>0\vert \mu_i')-\hat{p}_i\vert.
$$
Then, $\hat{\bm{\mu}}=\{\hat{\mu}_1,...,\hat{\mu}_n\}$ is our estimation of $\bm{\mu}$. We can further normalize $\hat{\bm{\mu}}$ in case that $\hat{\bm{\mu}}\not\in\mathcal{S}^{n-1}$.

We remark that this method can be hardly generalized to the case that $\kappa$ is unknown. The reason is that we do not have a similar strictly increasing structure, and that the bijection between the probability \eqref{von_density} and the parameter $\bm{\theta}$ might not exist.

\paragraph{Constraint Design for the Second Case}\

Now we discuss how to design $\bm{a}$ if we cannot control $b$. One difficulty that prevents us applying the same method as in the previous section is that we might have $(\bm{x}_t^*)_i=0$ while $(\bm{u}_t)_i>0$ due to an insufficient budget. However, if we know a lower bound of $\{b_t\}_{t=1}^{T}$, we can still estimate $\mathbb{P}(u_i>0\vert\bm{\mu})$ by dismantling the high-dimensional
estimation problem into a number of low-dimensional problems.

Denote the lower bound as $\underline{b}$. We set the first $\lceil\underline{b}\rceil$ entries in $\bm{a}$ be $1$ and others be $\infty$, where $\lceil\cdot\rceil$ denotes the ceiling function. In this case, we have 
$$
    (\bm{x}^*_t)_i>0 \Leftrightarrow (\bm{u}_t)_i>0 \text{ for all $i\leq\lceil\underline{b}\rceil$ and $t=1,...,T$. }
$$
In this way, we can estimate $\mathbb{P}(u_i\vert\bm{\mu})$ for $i=1,...,\lceil\underline{b}\rceil$ following the previous case. To estimate the probability for other $i>\lceil\underline{b}\rceil$, we can divide the problem into $n/\underline{b}$ parts and estimate the probabilities separately. 


\end{document}